\documentclass{amsart}

\usepackage{geometry, amssymb, esint, verbatim}
\usepackage{color}

\numberwithin{equation}{section}

\usepackage[colorlinks=true,urlcolor=blue, citecolor=red,linkcolor=blue,
linktocpage,pdfpagelabels, bookmarksnumbered,bookmarksopen]{hyperref}
\usepackage[hyperpageref]{backref}

\newtheorem{prop}{Proposition}[section]
\newtheorem{lm}[prop]{Lemma}

\newtheorem*{teoL}{Theorem L}
\newtheorem*{teoS}{Theorem S}

\theoremstyle{definition}
\newtheorem{oss}[prop]{Remark}
\newtheorem*{ack}{Acknowledgments}

\title[Singular orthotropic functionals]{Singular orthotropic functionals\\ with nonstandard growth conditions}
\author[Bousquet]{Pierre Bousquet}
\author[Brasco]{Lorenzo Brasco}
\author[Leone]{Chiara Leone}

\address[P. Bousquet]{Institut de Math\'ematiques de Toulouse, CNRS UMR 5219
\newline\indent Universit\'e de Toulouse
\newline\indent F-31062 Toulouse Cedex 9, France.}
\email{pierre.bousquet@math.univ-toulouse.fr}

\address[L.\ Brasco]{Dipartimento di Matematica e Informatica
	\newline\indent
	Universit\`a degli Studi di Ferrara
	\newline\indent
	Via Machiavelli 35, 44121 Ferrara, Italy}
\email{lorenzo.brasco@unife.it}

\address[C.\ Leone]{Dipartimento di Matematica ``R. Caccioppoli''
\newline\indent
Universit\`a degli Studi di Napoli ``Federico II''
\newline\indent
Via Cinthia, Complesso Universitario di Monte S. Angelo, 80126 Napoli, Italy}
\email{chiara.leone@unina.it}
\subjclass[2010]{35J75, 35B65, 49K20}
\keywords{Nonstandard growth conditions, singular elliptic equations, Lipschitz regularity, Sobolev regularity, orthotropic functionals}

\begin{document}

\begin{abstract}
We pursue the study of a model convex functional with orthotropic structure and nonstandard growth conditions, this time focusing on the sub-quadratic case. We prove that bounded local minimizers are locally Lipschitz. No restriction on the ratio between the highest and the lowest growth rates are needed. The result holds also in presence of a non-autonomous lower order term, under sharp integrability assumptions. Finally, we prove higher differentiability of bounded local minimizers, as well.
\end{abstract}

\maketitle

\begin{center}
\begin{minipage}{10cm}
\small
\tableofcontents
\end{minipage}
\end{center}

\section{Introduction}

\subsection{Overview}
In this paper we expand on the gradient regularity theory for minimizers of functionals from the Calculus of Variations, having an {\it orthotropic} structure, in the {\it nonstandard growth} case. This may be seen as a follow up of our previous papers \cite{BB2} and \cite{BBLV}.
\par
Specifically, for an open set $\Omega\subseteq\mathbb{R}^N$ and a set of exponents $1<p_1\le \dots\le p_N$, we take the anisotropic Sobolev space $W^{1,\mathbf{p}}_{\rm{loc}}(\Omega)$, defined by
$$
W^{1,\mathbf{p}}_{\rm{loc}}(\Omega)=\Big\{u\in L^1_{\rm{loc}}(\Omega)\, :\,  u_{x_i}\in L_{\rm{loc}}^{p_i}(\Omega),\, i=1,\cdots,N\Big\}.
$$
Given a function $f\in L^1_{\rm loc}(\Omega)$, we consider the following functional
\[
\mathfrak{F}_{\mathbf{p}}(u,\Omega'):=\sum_{i=1}^N \frac{1}{p_i}\,\int_{\Omega'} |u_{x_i}|^{p_i}\,dx-\int_{\Omega'} f\,u\,dx,\qquad \mbox{ for } u\in W^{1,\mathbf{p}}_{\rm loc}(\Omega)\cap L^\infty_{\rm loc}(\Omega)\ \mbox{ and } \Omega'\Subset \Omega.
\]
In the superquadratic case, i.e. for  
\[
2\le p_1\le\cdots \le p_N<\infty,
\]
and for $f\equiv 0$, it has been recently proved in \cite{BB2} that any local minimizer $U\in W^{1,\mathbf{p}}_{\rm loc}(\Omega)\cap L^\infty(\Omega)$ is such that
\[
\nabla U\in L^\infty_{\rm loc}(\Omega)\qquad \mbox{ and }\qquad |U_{x_i}|^\frac{p_i-2}{2}\,U_{x_i}\in W^{1,2}_{\rm loc}(\Omega),\ \mbox{ for } i=1,\dots,N.
\]
The main goal of this paper is to address the same kind of regularity issues, again for {\it bounded} local minimizers, this time in the subquadratic case
\[
1<p_1\le\cdots \le p_N\le 2.
\]
However, we will obtain some regularity results which actually hold in the full range $1<p_1\le \dots \le p_N<\infty$, see the next section for more details.
\par
We recall that $u\in W^{1,\mathbf{p}}_{\rm loc}(\Omega)\cap L^\infty_{\rm loc}(\Omega)$ is a {\it local minimizer of $\mathfrak{F}_{\mathbf{p}}$ in $\Omega$} if
\[
\mathfrak{F}_{\mathbf{p}}(u, \Omega') \le \mathfrak{F}_{\mathbf{p}}(\varphi, \Omega'),\qquad \hbox{ for every } \varphi-u\in W^{1,\mathbf{p}}_0(\Omega')\cap L^\infty(\Omega') \mbox{ and every }\Omega'\Subset\Omega.
\]
Here we denote by $W^{1,\mathbf{p}}_0(\Omega')$ the completion of $C^\infty_0(\Omega')$ with respect to the norm
\[
\varphi\mapsto \sum_{i=1}^N \|\varphi_{x_i}\|_{L^{p_i}(\Omega')},\qquad \mbox{ for every } \varphi\in C^\infty_0(\Omega').
\]
By convexity of $\mathfrak{F}_p$, we have that $u$ is a local minimizer  if and only if it is a local weak solution in $W^{1,\mathbf{p}}_{\rm loc}(\Omega)\cap L^\infty_{\rm loc}(\Omega)$ of the quasilinear equation
\[
-\sum_{i=1}^N \Big(|u_{x_i}|^{p_i-2}\,u_{x_i}\Big)_{x_i}=f,\qquad \mbox{ in }\Omega. 
\]
This can be seen as a particular instance of elliptic equation in the wide context of Musielak-Orlicz spaces, see \cite{CGSW} for a comprehensive study on the subject.
\par
We emphasize the fact that in this paper we just consider \emph{bounded} minimizers \(u\). As a consequence, we discard a priori all the counterexamples to regularity arising in the literature related to nonstandard growth variational problems, see \cite{Gia, Ho, MaContro}. For completeness, we mention that the boundedness of minimizers in this setting has already been extensively studied, see \cite{FS} for the homogeneous case \(f\equiv 0\) and \cite{Ci, Ci2} for the non-homogeneous one. 
 We also ignore the problem of the \emph{existence} of a  minimizer in \(W^{1, \mathbf{p}}\), for which we would also need to assume that \(f\) belongs to a suitable dual Sobolev space. Here instead, we assume {\it a priori} to have a bounded minimizer $U$ and focus on identifying sharp conditions on the function \(f\) needed to obtain its Lipschitz continuity and higher differentiability. 
 \par
The main feature of all our regularity results will be that {\it we do not need to impose any restriction on the ratio} $p_N/p_1$.
\vskip.2cm\noindent
We refer the reader to our previous papers \cite{BB2, BBJ, BBLV} for an introduction to the realm of gradient regularity for minimizers of orthotropic functionals (see also \cite{De} for an approach based on viscosity methods). We just recall here that already for the standard growth case $p_1=p_N=p$, the superquadratic case $p> 2$ is much more involved than the case of the model functional
\[
\mathfrak{L}(u;\Omega')=\int_{\Omega'} |\nabla u|^p\,dx,\qquad \mbox{ for } u\in W^{1,p}_{\rm loc}(\Omega)\ \mbox{ and } \Omega'\Subset \Omega,
\]
as far as the regularity of the gradient of local minimizers is concerned.
On the other hand, the subquadratic case $1<p<2$ is simpler, in a sense: the Lipschitz continuity is a consequence of a general result due to Fonseca and Fusco \cite[Theorem 2.2]{FF}, as observed in the introduction of \cite{BBJ}.  
\par
In particular, it seems natural to try to adapt the techniques used in \cite{FF} since they allow to establish the  Lipschitz regularity for the subquadratic case when $p_1=p_N=p<2$. 
However, we stress that in the case $p_1\not= p_N$, our functional pertains to the class of variational problems with nonstandard growth conditions, following the terminology of Marcellini in \cite{Ma89, Ma91}. Then it couples in a nontrivial way the difficulties coming from the two situations: orthotropic structure and nonstandard growth conditions. Thus, even if we will prove Lipschitz regularity with a proof inspired from \cite[Theorem 2.2]{FF}, nontrivial adaptations and intermediate results will be needed.
\par
Finally, it is worth recalling that, in spite of a large number of papers and contributions on nonstandard growth problems (including for example \cite{BL,BFZ, CDLL2,CDLL, ELM1, ELM, Ko, Kor, Leo, Pa, UU}), a complete gradient
regularity theory is still missing, even for the case of orthotropic structures. Moreover, we recall that also
the case of basic regularity (i.e. $C^{0,\alpha}$ estimates, Harnack inequalities, and an extension of the De Giorgi's regularity theory) is still not fully 
well-understood (see for example \cite{BCSV, BDP} and \cite{LS} for some results) for local minimizers of $\mathfrak{F}_{\mathbf{p}}$. 
\par
\subsection{Main results}

Our first result  is an higher integrability statement, which is valid without any restriction on the exponents $p_i$. As we will see in a while, this will be instrumental to the two main regularity results of this paper. In what follows, we  use the  notation
\begin{equation}
\label{G0}
\mathcal{G}_0(\nabla u)=\left(\sum_{i=1}^{N}\frac{1}{p_i}\,|u_{x_i}|^{p_i}-1\right)_++1,
\end{equation}
where $(\,\cdot\,)_+$ stands for the positive part. This function naturally arises from the  principal part of $\mathfrak{F}_{\mathbf{p}}$. 
It encodes in a natural way the full summability informations for each component of the gradient. A similar idea has been considered for example in the papers \cite{BCM, CM2, CM}, dealing with the so-called {\it double phase problems}.
\begin{prop}[Higher integrability: general growth]
\label{prop:high}
Let $1<p_1\le \dots \le p_N<\infty$ and let $f\in L^\gamma_{\rm loc}(\Omega)$, for some $\gamma\ge 2$. 
Then for every local minimizer $U\in W^{1,\mathbf{p}}_{\rm loc}(\Omega)\cap L^\infty_{\rm loc}(\Omega)$, we have
\[
\mathcal{G}_0(\nabla U)\in L^\gamma_{\rm loc}(\Omega).
\]
Moreover, for every ball $B_R(x_0)\Subset\Omega$ such that $B_{4R}(x_0)\Subset \Omega$, as well, we have
\begin{equation}
\label{apriori1}
\int_{B_\frac{R}{2}(x_0)}\mathcal{G}_0(\nabla U)^\gamma\,dx\le \left(\Gamma_1+\Gamma_2\,\int_{B_R(x_0)} \mathcal{G}_0(\nabla U)\,dx\right), 
\end{equation}
for two constants $\Gamma_1,\Gamma_2>0$ which depend only on 
\[
N,\, p_N,\, p_1,\,\gamma,\,R,\,\|f\|_{L^\gamma(B_{4R}(x_0))} \mbox{ and } \|U\|_{L^\infty(B_{4R}(x_0))}.
\]
\end{prop}
\begin{teoL}[Lipschitz regularity: subquadratic growth]
Let $1<p_1\le \dots \le p_N\le 2$ and let $f\in L^\gamma_{\rm loc}(\Omega)$, for some $\gamma>N$.
Then every local minimizer $U\in W^{1,\mathbf{p}}_{\rm loc}(\Omega)\cap L^\infty_{\rm loc}(\Omega)$ is locally Lipschitz continuous.
Moreover, for every ball \(B_R(x_0) \Subset \Omega\) such that $B_{4R}(x_0)\Subset\Omega$, we have
\[
\begin{split}
\left\|\mathcal{G}_0(\nabla U)\right\|_{L^{\infty}(B_{R/4}(x_0))}
 &\le C\,\left[\frac{1}{R^\frac{N\,\gamma}{\gamma-N}}\,\left(\Gamma_1+\Gamma_2\,\int_{B_R(x_0)} \mathcal{G}_0(\nabla U)\,dx\right)^\frac{N}{\gamma-N}+\|f\|_{L^\gamma(B_{R}(x_0))}^\frac{N\,\gamma}{\gamma-N}\right]\\
 &\times\left\|\mathcal{G}_0(\nabla U)\right\|_{L^1(B_{R}(x_0))}.
 \end{split}
 \]
for some $C = C(N,p_N,p_1,\gamma) > 0$ and for the same constants $\Gamma_1,\Gamma_2$ as in \eqref{apriori1}.
\end{teoL}
Observe that the assumption \(\gamma>N\) is sharp (in the scale of Lebesgue spaces) to obtain the Lipschitz continuity of \(U\). Actually, this is already true when \(p_1=\dots=p_N=2\). It is a remarkable fact that, even in the orthotropic case with nonstandard growth conditions, this universal assumption on $f$ still leads to Lipschitz continuity. We refer the reader to \cite{BM} for a wide class of variational problems (not including orthotropic structures, however) where this same condition is known to guarantee Lipschitz continuity of local minimizers. 
\begin{teoS}[Sobolev regularity: subquadratic growth]
Let $1<p_1\le \dots \le p_N\le 2$ and let 
\begin{equation}
\label{gammadif}
f\in L^{1+\frac{2}{p_1}}_{\rm loc}(\Omega).
\end{equation}
Then for every local minimizer $U\in W^{1,\mathbf{p}}_{\rm loc}(\Omega)\cap L^\infty_{\rm loc}(\Omega)$ we have
\[
\mathcal{W}_i:=|U_{x_i}|^\frac{p_i-2}{2}\,U_{x_i}\in W^{1,2}_{\rm loc}(\Omega),\ \mbox{ for } i=1,\dots,N \qquad \mbox{ and }\qquad \nabla U\in W^{1,\mathbf{p}}_{\rm loc}(\Omega).
\]
Moreover, for every ball \(B_R(x_0) \Subset \Omega\) such that $B_{4R}(x_0)\Subset\Omega$, we have for every $i=1,\dots,N$
\[
\begin{split}
\sum_{i=1}^N \int_{B_\frac{R}{4}(x_0)} \Big|\nabla \mathcal{W}_i\Big|^2 \,dx &\leq  \frac{C}{R^2}\,\left(\Gamma_1+\Gamma_2\,\int_{B_R(x_0)} \mathcal{G}_0(\nabla U)\,dx\right)+C\,R^{\frac{2}{p_1}-1}\,\int_{B_R(x_0)}|f|^{1+\frac{2}{p_1}}\,dx,
\end{split}
\]
and 
\[
\|\nabla U_{x_i}\|_{L^{p_i}(B_{R/4}(x_0))}\le \frac{2}{p_i}\,\|U_{x_i}\|_{L^{p_i}(B_{R/4}(x_0))}^{\frac{2-p_i}{2}}\, \|\nabla \mathcal{W}_i\|_{L^2(B_{R/4}(x_0))},
\]
for some $C=C(N,p_N,p_1)>0$ and for the same constants $\Gamma_1,\Gamma_2$ as in \eqref{apriori1}, corresponding to the choice $\gamma=(p_1+2)/p_1$.
\end{teoS}
\subsection{Comparison with known results} 
In the homogeneous case \(f=0\), Proposition \ref{prop:high} can be obtained as a  consequence of \cite[Lemma 4.2]{Li}. In the superquadratic case \(p_1\geq 2\) and still for \(f\equiv0\), an alternate proof can also be found in \cite[Proposition 6.1]{BB2}. We present here a new proof that takes into account the presence of the forcing term $f$. Our argument is certainly  more elementary than the one in \cite{BB2}, and arguably more natural in our setting than the one in \cite{Li}, in the sense that it strongly relies on some tools that will be repeatedly used in the proofs of the other main results, see the next section for further comments.
\vskip.2cm\noindent
Theorem L is the counterpart for the subquadratic case of our previous result \cite[Theorem 1.1]{BB2}, which deals with the superquadratic case. We shall explain in the next section why the two situations require different arguments.
We point out that experts in the field may recognize Theorem L (and \cite[Theorem 1.1]{BB2}, as well) as a particular case of the main result in \cite{Li}, at least in the homogeneous case \(f\equiv 0\). However, it turns out that the proof \cite[Proposition 2.1]{Li} is affected by a crucial flaw, we refer the reader to \cite[Remark 1.4]{BB2} for a detailed discussion on this delicate point. In any case, it is fair to admit that some other parts of Lieberman's paper \cite{Li} have been an important source of inspiration for the proof of  Proposition \ref{prop:high}.   
\par
In \cite[Corollary 3.4]{EMMP}, the authors proves the local Lipschitz continuity of local minimizers (not a priori bounded) of the following functional 
\begin{equation}
\label{EMMP}
\sum_{i=1}^N \frac{1}{p_i}\,\int (1+|u_{x_i}|^2)^{\frac{p_i}{2}}\,dx,\qquad \mbox{ for } \frac{2\,N}{N+2}<p_1\le \dots\le p_N\le 2.
\end{equation}
Observe that such a functional has an orthotropic structure, with nonstandard subquadratic growth conditions, exactly as our $\mathfrak{F}_p$. However, it should be noticed that the functional \eqref{EMMP} is {\it neither degenerate nor singular\,}: this is the crucial difference with our case.
Indeed, the Hessian of the function 
\[
\widetilde{G}(z):=\sum_{i=1}^N \frac{1}{p_i} (1+z_{i}^2)^{\frac{p_i}{2}},\qquad \mbox{ for } z=(z_1,\dots,z_N)\in\mathbb{R}^N,
\] 
satisfies
\[
(p_1-1)(1+|z|^2)^{\frac{p_1-2}{2}}\,|\xi|^2\leq \langle D^2 \widetilde{G} (z)\,\xi,\xi\rangle\leq |\xi|^2,\qquad \mbox{ for } z,\xi\in\mathbb{R}^N.
\]
This  property fails to be satisfied by our functional, where the integrand is given by 
\begin{equation}
\label{Gintro}
G(z)=\sum_{i=1}^N \frac{1}{p_i} |z_{i}|^{p_i},\qquad \mbox{ for } z\in\mathbb{R}^N.
\end{equation}
Even worse, in contrast with the general framework of \cite{EMMP}, in our situation there is no continuous functions \(h_1,h_2 : [0,+\infty) \to [0,+\infty)\) such that
\[
h_1(|z|)\,|\xi|^2\le \langle D^2 G (z)\,\xi,\xi\rangle \leq h_2(|z|)\,|\xi|^2,\qquad \mbox{ for }z,\xi\in\mathbb{R}^N,
\]
even for large values of \(z\). Indeed, $D^2 G(z)$ is given by the diagonal matrix
\[
D^2 G(z)= \left[\begin{array}{ccc}
(p_1-1)\,|z_1|^{p_1-2}& &\\
&\ddots&\\
& & (p_N-1)\,|z_N|^{p_N-2}
\end{array}\right],
\]
and each entry on the diagonal blows-up as the corresponding component of $z$ vanishes.
\vskip.2cm\noindent
As for Theorem S, we observe that this may be seen as a generalization of the following classical result for the $p-$Laplacian: for $1<p<2$, any local weak solution \(u\in W^{1, p}_{\rm loc}(\Omega)\) of 
$$
-\hbox{div}(|\nabla u|^{p-2}\nabla u) = f,\qquad \mbox{ in }\Omega,
$$
belong to \(W^{2,p}_{\rm loc}(\Omega)\), provided \(f\in L^{p'}_{\rm loc}(\Omega)\) (see for example \cite{Det}). The reader may notice that for $p_1=p_N=p$ our assumption \eqref{gammadif} boils down to
\[
f\in L^\gamma_{\rm loc}(\Omega),\qquad \mbox{ for } \gamma\ge 1+\frac{2}{p}.
\]
Since for $1<p<2$ we have $1+2/p<p'$, this is a weaker requirement when compared with the classical result recalled above. This is not surprising, since we are now assuming that $u$ is a priori bounded. Such an assumption is responsible for this new feature. In the standard growth case, this has been recently observed in \cite[Theorem 1.2]{CGP}.
\par
Higher differentiability of local minimizers is a well-studied problem: for the specific case of orthotropic functionals with subquadratic nonstandard growth, some prior results can be found for example in \cite[Theorem 3]{BL}, \cite[Corollary 1]{BL2} and \cite[Theorem 2]{CDLL}.
\par
Finally, in the superquadratic case $p_1\ge 2$, as already recalled the counterpart of Theorem S has been obtained in \cite[Corollary 7.1]{BB2}, for $f\equiv 0$. In the case of a right-hand side $f\not\equiv 0$, some results have been obtained in \cite[Theorem 1.1]{BLPV} and \cite[Corollary 2]{Pa}.

\begin{oss}[On the \(C^1\) regularity]
In dimension $N=2$, the $C^1$ regularity of a {\it Lipschitz} local minimizer essentially follows from \cite[Theorem 1.1]{DeSa}, both in the case \(p_1\leq 2\) and  \(p_1\geq 2\), provided that \(f\equiv 0\). This assertion is detailed in \cite{Bou}, where the ``mixed'' case \(p_1\leq 2\leq p_2\) is considered, as well. For a slightly different approach, see \cite{LR} when \(p_1\geq 2\) and   \cite{Ric} when \(p_1=p_2< 2\): these references still deal with the case $f\equiv 0$. In the non-homogenous case, the strategy followed in \cite{BB}  (and originally written for \(p_1=p_2\) and \(f\equiv 0\)) could be adapted to more general situations, provided \(f\) satisfies suitable differentiability and summability conditions. In any case, the \(C^1\) regularity of local minimizers when \(N\geq 3\) is entirely open, even for $p_1=p_N$ and \(f\equiv 0\).
\end{oss}

\subsection{Structure of the proofs}

The proofs of Proposition \ref{prop:high}, Theorem L and Theorem S are based on a classical three steps strategy. 
We first approximate our local minimizer $U\in W^{1,\mathbf{p}}_{\rm loc}(\Omega)\cap L^\infty_{\rm loc}(\Omega)$ by a sequence of minimizers $\{u^\varepsilon\}_{\varepsilon>0}$ of regularized functionals $\mathfrak{F}_{\mathbf{p},\varepsilon}$ having good smoothing properties. We next obtain uniform a priori estimates on these minimizers. 
Finally, we pass to the limit in order to transfer these a priori bounds to \(U\). 
\vskip.2cm\noindent
The first step is usually quite easy, it is sufficient to perturb the initial functional by adding some uniformly convex $\varepsilon-$perturbation and possibly smooth out the coefficients, for example by replacing $f$ with its mollifcation $f^\varepsilon$. This regularization strategy allows to avoid the usual difference quotient method, and the technicalities that go with its use in the nonstandard growth setting.
However, here we face a first difficulty: remember that we are not assuming $f$ to be in the correct dual Sobolev space. This also entails that we do not have a good a priori $L^\infty$ estimate at our disposal. Thus, such an approximation has to be handled with great care. We circumvent this technical difficulty, by adding a nonlinear lower term in the regularized functional, which forces the minimizers $u^\varepsilon$ to be bounded, with a uniform $L^\infty$ bound which only depends on the local $L^\infty$ norm of $U$ (see Lemma \ref{lm:below}). This is a technical aspect of the proof, which we believe to have its own interest.
\vskip.2cm\noindent
The core of the matter is next to establish the a priori estimates for the gradient of \(u^{\varepsilon}\), the minimizer of the regularized functional \(\mathfrak{F}_{\mathbf{p},\varepsilon}\). 
As for the estimates leading to Proposition \ref{prop:high} and Theorem L, these are achieved by means of Moser--type schemes: a {\it slow one} and a {\it fast one}, respectively.
\par
The cornerstone of these schemes is a Caccioppoli inequality for power functions of the gradients (see Proposition \ref{prop:caccioespilon2}). In a simplified way, for every $\alpha\ge 0$ this reads as
\begin{equation}\label{eq213}
\sum_{i=1}^N\int |u_{x_i}|^{p_i-2}\,\left|\Big(G(\nabla u)^\frac{\alpha+1}{2}\Big)_{x_i}\right|^2\,dx\lesssim \alpha^2\int G(\nabla u)^{\alpha+2-\frac{2}{p_N}}\,dx,
\end{equation}
where $G$ is the same function as in \eqref{Gintro}. 
For simplicity, we put $f\equiv 0$ and write $u$ in place of $u^\varepsilon$.
\par
Observe that on the left-hand side of \eqref{eq213} we have a weighted gradient of a power of $G(\nabla u)$: the weights \(|u_{x_i}|^{p_i-2}\) are the typical feature of degenerate/singular orthotropic functionals. The main difficulty in getting regularity results out of this estimate is precisely due to their presence.
In contrast with the Caccioppoli inequality previously obtained in \cite[Lemma 3.1]{BB2} to handle the superquadratic case $p_1\ge 2$,
now these weights do not pop-up on the right-hand side. This is a crucial ingredient of the estimate: indeed, no
control from above would be possible on \(|u_{x_i}|^{p_i-2}\) if \(p_i<2\). 
\par
Not surprisingly, the proof of \eqref{eq213} relies on the differentiated Euler-Lagrange equation, which is nothing but the equation solved by the components of $\nabla u^\varepsilon$. In a nutshell, the idea to reach such an estimate not containing the nasty weights \(|u_{x_i}|^{p_i-2}\) on the right-hand side, is that of using an integration by parts trick: this permits to trade the presence of the term $D^2 G$,
with the more tractable one $D G$. This idea is certainly not new in the context of singular variational problems: it goes back at least to \cite{Na}, and has then become standard in the field. 
\par
However, as natural as this idea may appear, its technical implementation in our context needs some efforts: in particular, a careful choice of the test functions for the differentiated equation has to be done. Such a choice must reflect the algebraic structure of the operator, in a sense. Without entering too much into the details, we refer to the proof of Proposition \ref{prop:caccioespilon2} below. The choice of the correct test functions here has been suggested to us by \cite{Li}, even if our choice seems to be simpler and more natural.
\vskip.2cm\noindent
The Caccioppoli inequality \eqref{eq213} is first used in the proof of the higher integrability result of Proposition \ref{prop:high}. More specifically, it permits to obtain a self-improving estimate of the type
\begin{equation}\label{eq242}
\int G(\nabla u)^{\beta + 1}\,dx \lesssim \|u\|_{L^{\infty}}^2\,\beta^2\, \int G(\nabla u)^{\beta+1-\frac{2}{p_N}}\,dx.
\end{equation}
This is the {\it slow} Moser's iteration scheme we were referring to above: by iterating \eqref{eq242} a finite number of times, we can conclude that $G(\nabla u)$ (and thus $\nabla u$ itself) 
can be estimated in \(L^{q}\), for every {\it finite} \(q\geq 1\). Observe that the {\it additive} integrability gain at each step and the presence of the factor $\beta^2$ on the right-hand side make the previous scheme not suitable for being iterated infinitely many times.
 This explains why we cannot reach the limiting case \(\nabla u \in L^{\infty}\) with this approach. 
\par
Estimates like \eqref{eq242} are quite typical in the Regularity Theory, both in the contexts of standard and nonstandard growth problems (among others, see for example \cite[Proposition 3.1]{DiB} and \cite[Theorem 3]{Ch}, respectively). Usually, they are obtained by coupling an integration by parts, with a Caccioppoli inequality for the gradient, like the one \eqref{eq213} at our disposal. The $L^\infty$ bound on the solution is used to treat the solution itself as a constant in the estimates.  
\par
We stress here that in this part of the proofs we do not need the restriction $p_N\le 2$. Thus, in particular we can extend and simplify the higher integrability result we previously obtained\footnote{There is however a subtle detail here: the result in \cite{BB2} was obtained through a complicate self-improving iterative scheme (inspired from that of \cite[Theorem 1.1]{BFZ}), which {\it was not} of Moser--type. Actually, this was much more sophisticated and could be roughly described as follows: improvement of integrability of $N-1$ components of the gradient entails that the missing one improves its integrability, as well.} in \cite[Proposition 4.3]{BB2}.
\vskip.2cm\noindent
With the aid of \eqref{eq242}, in the case $p_N\le 2$ we can transpose to our situation the typical {\it absorption trick} which lies at the basis of the Lipschitz estimate for the standard $p-$Laplacian (see for example \cite[Section 3]{DiB}). Up to some nontrivial technical issues, this consists in observing that when \(G(\nabla u)\geq 1\), we have
\[
|u_{x_i}|^{p_i-2}\gtrsim  G(\nabla u)^{\frac{p_1-2}{p_1}}, \qquad \mbox{ since } p_i-2\le 0,
\]
and thus
\begin{equation}
\label{eq2421}
\sum_{i=1}^N\int |u_{x_i}|^{p_i-2}\,\left|\Big(G(\nabla u)^\frac{\alpha+1}{2}\Big)_{x_i}\right|^2\,dx\gtrsim \sum_{i=1}^N \int \left|\left(G(\nabla u)^{\frac{\alpha+2}{2}-\frac{1}{p_1}}\right)_{x_i}\right|^2\,dx.
\end{equation}
The weights $|u_{x_i}|^{p_i-2}$ have then been absorbed into a suitable power function of the gradient. In this sense, in the case $p_N\le 2$, the presence of the weights $|u_{x_i}|^{p_i-2}$ on the left-hand side of \eqref{eq242} helps, more than it hurts. 
\par
At this point, by joining \eqref{eq2421} and \eqref{eq242}, the orthotropic nature of the problem completely disappears and we simply fall into the realm of nonstandard growth problems.
A standard application of Sobolev inequality makes then possible to launch a standard Moser's iterative scheme (i.e. a {\it fast} one, with a {\it multiplicative} gain of integrability at each step). 
This permits to reach an $L^\infty-L^q$ estimate on $G(\nabla u)$, after infinitely many iterations.
This is not the end of the story. Indeed, we still have to pay attention to a detail which is quite typical of the nonstandard growth case: 
 the exponent $q$ in this a priori estimate could be too large. However, this preliminary estimate can be ``rectified'' by combining the higher integrability result of Proposition \ref{prop:high} together with an interpolation trick which  decreases the initial integrability requirement on $G(\nabla u)$. We then finally get a $L^\infty-L^1$ estimate on $G(\nabla u)$, as desired. 
\vskip.2cm\noindent
In contrast to Proposition \ref{prop:high} and Theorem L, the proof of Theorem S does not rely on the Caccioppoli inequality of Proposition \ref{prop:caccioespilon2}. The proof follows the same idea as in the case of the result for the familiar $p-$Laplacian, for the case $1<p<2$: we test the differentiated equation with the gradient itself $u_{x_k}$ and perform an integration by parts as in Naumann's trick \cite{Na}. Again, this permits to avoid using the undesired upper bound on the Hessian $D^2 G$. In order to conclude, one has to control from above terms of the form
\[
u_{x_kx_k}\,|u_{x_i}|^{p_i-1}.
\]
Observe that for every $k\not=i$, the two terms are completely decoupled. However, when $p_i\le 2$, we can simply estimate this term from above by Young's inequality  
\[
u_{x_kx_k}\,|u_{x_i}|^{p_i-1} \lesssim |u_{x_k}|^{p_k-2}\,|u_{x_kx_k}|^2 +|u_{x_k}|^{2-p_k}\,|u_{x_i}|^{2\,(p_i-1)}.
\] 
The first term is absorbed on the left-hand side, while the second term can be estimated from above by means of an integrability estimate (here we rely again on the information provided by Proposition \ref{prop:high}). This explains why we require \(p_N\leq 2\) in the statement of Theorem S. 

\subsection{Plan of the paper}
In Section \ref{preli}, we present the approximation scheme and some basic material used all along the paper. Section \ref{caccio} contains the crucial Caccioppoli-type inequality for the gradient (Proposition \ref{prop:caccioespilon2}). The latter is exploited in Section \ref{sec:uniform} to perform the slow Moser iteration   leading to the higher integrability estimate needed in  Proposition \ref{prop:high}. The Lipschitz bound related to Theorem L is proved in Section \ref{sec:Lip}, while Section \ref{diff} is devoted to the proof of the higher differentiability estimates corresponding to Theorem S. Then, in Section \ref{proofmain}, we eventually prove our three main results by passing to the limit in the approximation scheme. 
Finally, for completeness, we include in Appendix \ref{appA} the proof of a maximum principle ensuring  the uniform boundedness of the approximating sequence. 

\begin{ack}
This work has been finalized during a staying of P.\,B. and L.\,B. at the {\it Institute of Applied Mathematics and Mechanics} of the University of Warsaw, in July 2022. Iwona Chlebicka and Anna Zatorska-Goldstein are gratefully acknowledged for their kind invitation and the nice working atmosphere provided during the whole staying.
\par
C.\,L. is a member of the Gruppo Nazionale per l'Analisi Matematica, la Probabilit\`a
e le loro Applicazioni (GNAMPA) of the Istituto Nazionale di Alta Matematica (INdAM).
The three authors gratefully acknowledge the financial support of the projects FAR 2019 and FAR 2020 of the University of Ferrara.
\end{ack}

\section{Preliminaries}\label{preli}
In this section, we fix 
\[
1<p_1\le p_2\le \dots\le p_N<\infty.
\]
\subsection{Some auxiliary functions}
For every $i=1,\dots,N$ and $\varepsilon>0$, we define
\begin{equation}
\label{gi}
g_{i,\varepsilon}(t)=\frac{1}{p_i}\,(\varepsilon+t^2)^\frac{p_i}{2},\qquad \mbox{ for every } t\in\mathbb{R}.
\end{equation}
\begin{lm}[Sub-quadratic case]
Let $1<p_i\le 2$. Then for every $t\in\mathbb{R}$, we have
\begin{equation}
\label{g2}
(p_i-1)\,(\varepsilon+t^2)^\frac{p_i-2}{2}\le g''_{i,\varepsilon}(t)\le (\varepsilon+t^2)^\frac{p_i-2}{2},
\end{equation}
\begin{equation}
\label{g3}
g_{i,\varepsilon}(t)\le \frac{1}{p_i}\,\left(\varepsilon^\frac{p_i}{2}+g_{i,\varepsilon}'(t)\,t\right),
\end{equation}
\begin{equation}
\label{g4}
|g_{i,\varepsilon}'(t)|^2\le \frac{p_i}{p_i-1}\,g_{i,\varepsilon}''(t)\,g_{i,\varepsilon}(t).
\end{equation}
\end{lm}
\begin{proof}
The second derivative of $g_{i,\varepsilon}$ is given by
\[
\begin{split}
g''_{i,\varepsilon}(t)&=(\varepsilon+t^2)^\frac{p_i-2}{2}+(p_i-2)\,(\varepsilon+t^2)^\frac{p_i-4}{2}\,t^2=(\varepsilon+t^2)^\frac{p_i-2}{2}\,\left(1+(p_i-2)\,\frac{t^2}{\varepsilon+t^2}\right).
\end{split}
\]
In particular, by using that $0<p_i-1\le 1$, we easily get \eqref{g2}.
\vskip.2cm\noindent
We also observe that 
\[
\begin{split}
g_{i,\varepsilon}(t)=\frac{1}{p_i}\,(\varepsilon+t^2)^\frac{p_i}{2}
&=\frac{1}{p_i}\,(\varepsilon+t^2)^\frac{p_i-2}{2}\,\varepsilon+\frac{1}{p_i}\,(\varepsilon+t^2)^\frac{p_i-2}{2}\,t^2\le \frac{\varepsilon^\frac{p_i}{2}}{p_i}+\frac{g_{i,\varepsilon}'(t)}{p_i}\,t,
\end{split}
\]
which proves \eqref{g3}.
\vskip.2cm\noindent
Finally, \eqref{g4} follows by writing 
\[
\begin{split}
|g'_{i,\varepsilon}(t)|^2&=(\varepsilon+t^2)^\frac{p_i-2}{2}\,(\varepsilon+t^2)^\frac{p_i-2}{2}\,t^2\le (\varepsilon+t^2)^\frac{p_i-2}{2}\,(\varepsilon+t^2)^\frac{p_i}{2},
\end{split}
\]
and then using the definition of $g_{i,\varepsilon}$ and the lower bound in \eqref{g2}.
\end{proof}
\begin{lm}[Super-quadratic case]
Let $p_i> 2$. Then for every $t\in\mathbb{R}$, we have
\begin{equation}
\label{g2b}
(\varepsilon+t^2)^\frac{p_i-2}{2}\le g''_{i,\varepsilon}(t)\le (p_i-1)\,(\varepsilon+t^2)^\frac{p_i-2}{2},
\end{equation}
\begin{equation}
\label{g3b}
g_{i,\varepsilon}(t)\le \frac{2^\frac{p_i-2}{2}}{p_i}\,\left(\varepsilon^\frac{p_i}{2}+g_{i,\varepsilon}'(t)\,t\right),
\end{equation}
\begin{equation}
\label{g4b}
|g_{i,\varepsilon}'(t)|^2\le p_i\,g_{i,\varepsilon}''(t)\,g_{i,\varepsilon}(t).
\end{equation}
\end{lm}
\begin{proof}
The proofs of \eqref{g2b} and \eqref{g4b} are similar to those of \eqref{g2} and \eqref{g4} respectively and we omit them.
In order to prove \eqref{g3b}, we use the convexity of the map $\tau\mapsto|\tau|^{p_i/2}$. This implies
\begin{equation}\label{eq464}
g_{i,\varepsilon}(t)=\frac{1}{p_i}\,(\varepsilon+t^2)^\frac{p_i}{2}\le \frac{2^\frac{p_i-2}{2}}{p_i}\,\left(\varepsilon^\frac{p_i}{2}+|t|^{p_i}\right).
\end{equation}
We then observe that 
\[
|t|^{p_i}=|t|^{p_i-2}\,t^2\le (\varepsilon+t^2)^\frac{p_i-2}{2}\,t^2=g_{i,\varepsilon}'(t)\,t.
\]
By combining the last two inequalities, we get \eqref{g3b}.
\end{proof}
We also define the function
\begin{equation}
\label{G}
G_\varepsilon(z):=\sum_{i=1}^N g_{i,\varepsilon}(z_i)=\sum_{i=1}^N \frac{1}{p_i}(\varepsilon+z_i^2)^\frac{p_i}{2},
\end{equation}
which will play a crucial role in our estimates. The next result holds without any restriction on $p_i$.
\begin{lm}
\label{lm:majorette}
For every $z=(z_1,\dots,z_N)\in\mathbb{R}^N$ and every $i=1,\dots,N$, we have
\begin{equation}
\label{majorette}
|g'_{i,\varepsilon}(z_i)|\le p_N^\frac{p_N-1}{p_N}\,G_\varepsilon(z)^\frac{p_i-1}{p_i}.
\end{equation}
\end{lm}
\begin{proof}
By recalling the definition of both $g_{i,\varepsilon}$ and $G_{\varepsilon}$, we have 
\[
\begin{split}
|g'_{i,\varepsilon}(z_i)|\le (\varepsilon+|z_i|^2)^\frac{p_i-1}{2}&=\big(p_i\,g_{i,\varepsilon}(z_i)\big)^\frac{p_i-1}{p_i}\le\left(\sum_{k=1}^N p_k\,g_{k,\varepsilon}(z_k)\right)^\frac{p_i-1}{p_i}\le p_N^\frac{p_i-1}{p_i}\,G_\varepsilon(z)^\frac{p_i-1}{p_i}.
\end{split}
\]
By using that $p_N>1$ and that $p_i\le p_N$, we can estimate the last term from above as claimed.
\end{proof}

\subsection{Regularized problems}
We will use an approximation scheme which is similar to that already used in our previous papers, starting from \cite[Section 2]{BBJ}.
We want to consider local minimizers of the following convex integral functional
\[
\mathfrak{F}_{\mathbf{p}}(u;\Omega')=\sum_{i=1}^N \frac{1}{p_i}\,\int_{\Omega'} |u_{x_i}|^{p_i}\, dx-\int_{\Omega'} f\,u\,dx,\qquad u\in W^{1,\mathbf{p}}_{\rm loc}(\Omega)\cap L^\infty_{\rm loc}(\Omega),\ \Omega'\Subset\Omega.
\]
The function $f$ is taken to belong to $L^1_{\rm loc}(\Omega)$.
In the rest of the paper, {\it we fix $U\in W^{1,\mathbf{p}}_{\rm loc}(\Omega)\cap L^\infty_{\rm loc}(\Omega)$ a bounded local minimizer of} $\mathfrak{F}_\mathbf{p}$. We also fix a ball 
\[
B \Subset \Omega\quad \mbox{ such that }\quad 2\,B\Subset\Omega \mbox{ as well}.
\] 
Here by $\lambda\,B$ we denote the ball concentric with $B$, scaled by a factor $\lambda>0$. 
We set
\[
\varepsilon_0=\min\left\{1, \frac{\mbox{ radius of }B}{2}\right\}>0.
\]
For every $0<\varepsilon\le \varepsilon_0$ and every \(x\in \overline{B}\), we then define 
\[
U^\varepsilon(x)=U\ast \varrho_\varepsilon(x)\qquad \mbox{ and }\qquad f^\varepsilon(x)=f\ast \varrho_\varepsilon(x).
\] 
As usual, we denote by $\varrho_\varepsilon$ the usual family of Friedrichs mollifiers, supported in a ball of radius $\varepsilon$ centered at the origin. 
Finally, we set
\[
M:=\|U\|_{L^\infty(2\,B)},
\]
and
take $\zeta^\varepsilon$ to be a $C^\infty$ function of one variable, such that 
\[
0\le (\zeta^\varepsilon)'\le 1\qquad \mbox{ and }\qquad \zeta^\varepsilon(t)=\left\{\begin{array}{cc}
M,& \mbox{ if } t\ge M+\varepsilon,\\
-M,& \mbox{ if } t\le -M-\varepsilon,
\end{array}
\right.
\]
and such that
\[
\lim_{\varepsilon\to 0} \|\zeta^\varepsilon-\zeta\|_{L^\infty(\mathbb{R})}=0,\qquad \mbox{ where } \zeta(t)=\max\Big\{\min\{t,M\},-M\Big\}, \mbox{ for } t\in\mathbb{R}.
\]
By recalling the definition \eqref{gi} of $g_{i,\varepsilon}$, we then define the regularized functional
\[
\mathfrak{F}_{\mathbf{p},\varepsilon}(v;B)=\sum_{i=1}^N \int_B g_{i,\varepsilon}(v_{x_i})\, dx-\int_B f^\varepsilon\,\zeta^\varepsilon(v)\,dx.
\]
\begin{lm}[Existence and regularity of a minimum for $\mathfrak{F}_{\mathbf{p},\varepsilon}$]
\label{lm:below}
For every $0<\varepsilon\le \varepsilon_0$, the problem
\begin{equation}
\label{approximated}
\min\left\{\mathfrak{F}_{\mathbf{p},\varepsilon}(v;B)\, :\, v-U^\varepsilon\in W^{1,\mathbf{p}}_0(B)\right\},
\end{equation}
admits a solution $u^\varepsilon$, which belongs to $C^{\infty}(\overline{B})$.
Moreover, for every \(0<\varepsilon\le \varepsilon_0\), we have
\[
\|u^\varepsilon\|_{L^\infty(B)}\le M+\varepsilon.
\]
\end{lm}
\begin{proof}
We first show that we can apply \cite[Theorem 9.2]{St} and get existence of a solution to 
\begin{equation}
\label{approximatedLip}
\min\left\{\mathfrak{F}_{\mathbf{p},\varepsilon}(v;B)\, :\, v-U^\varepsilon\in W^{1,\infty}_0(B)\right\}.
\end{equation}
For this, we check the required assumptions. We first claim that for every \(z,\xi\in\mathbb{R}^N\)
\begin{equation}\label{eq363}
\sum_{i=1}^{N}g_{i, \varepsilon}''(z_i)\,\xi_i^2\geq \nu_{\varepsilon}\, (1+|z|^2)^{\tau}\,|\xi|^2, \qquad \mbox{ with }\left\{\begin{array}{l}
\nu_\varepsilon:=\min \left\{p_1-1, \varepsilon^{\frac{p_N-2}{2}}\right\}>0,\\
\\
\tau:=\min\left\{0, \dfrac{p_1-2}{2}\right\}>-\dfrac{1}{2}.
\end{array}
\right.
\end{equation}
From \eqref{g2} and \eqref{g2b} we get
\[
\sum_{i=1}^{N}g_{i, \varepsilon}''(z_i)\,\xi_i^2\geq \sum_{i=1}^N \min\{p_i-1,1\}\,(\varepsilon+z_i^2)^{\frac{p_i-2}{2}}\,\xi_i^2.
\]
Then we observe that if \(p_i\le 2\), then \(p_1\le 2\) and we can write:
\[
\min\{p_i-1,1\}\,(\varepsilon+z_i^2)^{\frac{p_i-2}{2}}\ge (p_1-1)\, (1+|z|^2)^\frac{p_1-2}{2}=(p_1-1)\, (1+|z|^2)^\tau.
\]
If instead \(p_i>2\), then
\[
\min\{p_i-1,1\}\,(\varepsilon+z_i^2)^{\frac{p_i-2}{2}}\ge \varepsilon^\frac{p_N-2}{2}\ge \,\varepsilon^\frac{p_N-2}{2}\, (1+|z|^2)^\tau.
\]
Thus, in both cases, \eqref{eq363} holds.
\par
As for the lower order term, observe that the smooth function \(h^\varepsilon(x,u):=f^{\varepsilon}(x)\,\zeta^\varepsilon(u)\) verifies
\[
\frac{\partial h^{\varepsilon}}{\partial u}(x,u)=0,\qquad \mbox{ for every } (x,u)\in B\times (\mathbb{R}\setminus (-M-\varepsilon,M+\varepsilon)).
\] 
Finally, the uniform convexity of \(B\) and the smoothness of \(U^\varepsilon\) entail that the latter satisfies the bounded slope condition.

Then \cite[Theorem 9.2]{St} yields the existence of a solution to \eqref{approximatedLip}. Since all the data are smooth, \cite[Theorem 9.3]{St} implies that \(u^\varepsilon\in C^{\infty}(\overline{B})\). 
We claim that \(u^\varepsilon\) is a solution of \eqref{approximated}, as well. Indeed, 
by using \cite[Theorem 1.1]{BMT}, for every \(v\in U^\varepsilon+W^{1,\mathbf{p}}_0(B)\) we can infer the existence of  \(\{v^k\}_{k\in\mathbb{N}}\subset U^\varepsilon+W^{1,\infty}_0(B)\) converging to \(v\) in \(W^{1,p_1}(B)\) and such that 
\[
\lim_{k\to \infty} \sum_{i=1}^N \int_B g_{i,\varepsilon}(v^{k}_{x_i})\, dx = \sum_{i=1}^N \int_B g_{i,\varepsilon}(v_{x_i})\, dx. 
\]
By the dominated convergence theorem and the uniform boundedness of \(\zeta^{\varepsilon}\), we also have
\[
\lim_{k\to \infty}\int_B f^\varepsilon\, \zeta^\varepsilon(v^k)\,dx = \lim_{k\to \infty}\int_B f^\varepsilon\, \zeta^\varepsilon(v)\,dx.
\] 
This proves that there is no Lavrentiev phenomenon for \(\mathfrak{F}_{\mathbf{p},\varepsilon}\), that is
\[
\mathfrak{F}_{\mathbf{p},\varepsilon}(u^\varepsilon ;B) = \min_{U^\varepsilon+W^{1,\infty}_0(B)} \mathfrak{F}_{\mathbf{p},\varepsilon} = \inf_{U^\varepsilon+W^{1,\mathbf{p}}_0(B)} \mathfrak{F}_{\mathbf{p},\varepsilon}.
\]
Thus, we get that \(u^\varepsilon\) solves \eqref{approximated}, as well.
Finally, the claimed \(L^{\infty}\) estimate readily follows from Lemma \ref{lm-bd-1} in the Appendix.
\end{proof}
The smooth minimizer \(u^\varepsilon\) satisfies the Euler-Lagrange equation
\begin{equation}
\label{nondiff}
\sum_{i=1}^N \int g_{i,\varepsilon}'(u^\varepsilon_{x_i})\,\varphi_{x_i}\,dx=\int f^{\varepsilon} (\zeta^{\varepsilon})'(u^\varepsilon)\,\varphi\,dx, \qquad \mbox{ for every } \varphi\in W^{1,\mathbf{p}}_0(B).
\end{equation}

For every \(k=1, \dots, N\), one can insert test functions of the form \(\varphi_{x_k}\), with $\varphi\in C^2$ compactly supported in $B$. By integrating by parts, we then get the equation for the partial derivatives of \(u^\varepsilon\) 
\begin{equation}
\label{Eulerdiff}
\sum_{i=1}^N \int g_{i,\varepsilon}''(u^\varepsilon_{x_i})\,u^{\varepsilon}_{x_k\,x_i}\varphi_{x_i}\,dx=\int \big(f^{\varepsilon}\,(\zeta^{\varepsilon})'\,(u^{\varepsilon})\big)_{x_k}\,\varphi\,dx,\qquad \mbox{ for every } k=1,\dots,N.
\end{equation}
As usual, by a density argument, the equation can be tested by any $\varphi\in W^{1,\mathbf{p}}_0(B)$.

The first ingredient of our recipe is a simple  a priori estimate, which is essentially the same as in \cite[Lemma 2.1]{BB2}: the only difference is the presence of the non-autonomous  and nonlinear term $f^\varepsilon\,\zeta^\varepsilon(v)$, together with a slight modification of the function $g_{i,\varepsilon}$.
\begin{lm}[Basic energy estimate]
\label{lm:unif}
For every $0<\varepsilon\le \varepsilon_0$, the following uniform estimate holds 
\[
\sum_{i=1}^N \frac{1}{p_i}\,\int_B |u^\varepsilon_{x_i}|^{p_i}\, dx\le C\,\left(\sum_{i=1}^N \frac{1}{p_i}\,\int_{2\,B} |U_{x_i}|^{p_i}\, dx+\varepsilon^\frac{p_1}{2}+\|f\|_{L^1(2\,B)}\right),
\]
for some $C=C(N,B,\mathbf{p},M)>0$.
\end{lm}
\begin{proof}
By testing the minimality of $u^\varepsilon$ against $U^\varepsilon$, we obtain
\[
\begin{split}
\sum_{i=1}^N \int_B g_{i,\varepsilon}(u^\varepsilon_{x_i})\, dx\le \sum_{i=1}^N \int_B g_{i,\varepsilon}(U^\varepsilon_{x_i})\, dx-\int_B f^\varepsilon\,\big(\zeta^\varepsilon(U^\varepsilon)-\zeta^\varepsilon(u^\varepsilon)\big)\,dx.
\end{split}
\]
The convexity of the function $g_{i,\varepsilon}$ allows to apply Jensen's inequality in connection with the fact that $U^\varepsilon$ is defined by a convolution. This gives
\[
\int_B g_{i,\varepsilon}(U^\varepsilon_{x_i})\, dx\le \int_B  g_{i,\varepsilon}(U_{x_i})\ast \varrho_\varepsilon\, dx\le \int_{2\,B} g_{i,\varepsilon}(U_{x_i})\,dx.
\]
By  using also that $g_{i,\varepsilon}(t)\ge |t|^{p_i}/p_i$ and the $1-$Lipschitz character of $\zeta^\varepsilon$, we get
\[
\begin{split}
\sum_{i=1}^N \frac{1}{p_i}\,\int_B |u^\varepsilon_{x_i}|^{p_i}\, dx&\le \sum_{i=1}^N \int_{2\,B} g_{i,\varepsilon}(U_{x_i})\,dx+\int_B |f^\varepsilon|\,|U^\varepsilon-u^\varepsilon|\,dx\\
&\le \sum_{i=1}^N \int_{2\,B} g_{i,\varepsilon}(U_{x_i})\,dx+C_M\,\|f\|_{L^1(2B)},
\end{split}
\]
where \(C_M\) is a positive constant which only depends on \(M\).
We finally rely on  \eqref{eq464} when \(p_i>2\) or the subadditivity of \(t\mapsto |t|^{\frac{p_i}{2}}\) when \(p_i\le 2\) to obtain 
\begin{equation}
\label{piccolina}
g_{i,\varepsilon}(t)\le \frac{\max\left\{1,2^{\frac{p_i-2}{2}}\right\}}{p_i}\,({\varepsilon}^{\frac{p_i}{2}}+|t|^{p_i}).
\end{equation}
This concludes the proof.
\end{proof}

In view of our scopes, it is mandatory to have a convergence result for the minimizers $\{u^\varepsilon\}_{0<\varepsilon<\varepsilon_0}$. This is the content of the next lemma, which is an extension of \cite[Lemma 2.2]{BB2}.
\begin{lm}[Convergence to a minimizer]
\label{lm:convergence}
With the same notation as above, we have
\begin{equation}
\label{troppoforte!}
\lim_{\varepsilon\to 0} \left[\|u^{\varepsilon}-U\|_{L^{q}(B)}+\sum_{i=1}^N \|(u^\varepsilon-U)_{x_i}\|_{L^{p_i}(B)}\right]=0,
\end{equation}
for every $1\le q< \infty$.
\end{lm}
\begin{proof}
The proof goes as in \cite[Lemma 2.2]{BB2}. We repeat the argument, since this gives us the occasion to fix some missing details in \cite{BB2}.
By using the uniform estimate of Lemma \ref{lm:unif} and the definition of $U^\varepsilon$, we get that $\{u^\varepsilon-U^\varepsilon\}_{0<\varepsilon\le \varepsilon_0}$ is a bounded family in $W^{1,\mathbf{p}}_0(B)$. 
Thanks to Lemma \ref{lm:below}, we also have that $\{u^\varepsilon-U^{\varepsilon}\}_{0<\varepsilon\le \varepsilon_0}$ is a bounded family in $L^\infty(B)$.
From those two facts, we can infer the existence of an infinitesimal sequence $\{\varepsilon_k\}_{k\in\mathbb{N}}\subset(0,\varepsilon_0]$ such that 
$\{u^{\varepsilon_k}-U^{\varepsilon_k}\}_{k\in \mathbb{N}}\) converges weakly in $W_0^{1,\mathbf{p}}(B)$ and almost everywhere to a function $\phi\in W_0^{1,\mathbf{p}}(B)\cap L^\infty(B)$,  in the sense that 
\[
\lim_{k\to\infty} \int_B (u^{\varepsilon_k}-U^{\varepsilon_k})\,\varphi\,dx=\int_B \phi\,\varphi\,dx,\qquad \mbox{ for every } \varphi\in L^1(B),
\] 
and
\[
\lim_{k\to\infty} \int_B (u^{\varepsilon_k}-U^{\varepsilon_k})_{x_i}\,\varphi\,dx=\int_B \phi_{x_i}\,\varphi\,dx,\qquad \mbox{ for every }\varphi\in L^{p_i'}(B), i=1,\dots, N.
\]
By recalling that $U^{\varepsilon_k}$ has been constructed by convolution, we also have that it converges strongly in $W^{1,\mathbf{p}}(B)$ and almost everywhere to $U$. This permits to conclude that $\{u^{\varepsilon_k}\}_{k\in\mathbb{N}}$ converges weakly and almost everywhere to $u:=\phi+U$. 
We need to prove that actually $u=U$. With this aim, we test the minimality of each $u^{\varepsilon_k}$ against the function $U^{\varepsilon_k}$. Thus, by lower semicontinuity of the $L^{p_i}$ norms, we can infer 
\begin{equation}
\label{maroc}
\begin{split}
\sum_{i=1}^N\frac{1}{p_i}\, \int_B |u_{x_i}|^{p_i}\, dx&\le\liminf_{k\to \infty}\sum_{i=1}^N\frac{1}{p_i}\, \int_B |u^{\varepsilon_k}_{x_i}|^{p_i}\, dx\\
&\le \liminf_{k\to \infty}\sum_{i=1}^N\frac{1}{p_i}\, \int_B g_{i,\varepsilon_k}(u^{\varepsilon_k}_{x_i})\, dx\\
&\le \lim_{k\to \infty}\sum_{i=1}^N \frac{1}{p_i}\,\int_B g_{i,\varepsilon_k}(U^{\varepsilon_k}_{x_i})\, dx-\int_B f^{\varepsilon_k}\,\big(\zeta^{\varepsilon_k}(U^{\varepsilon_k})-\zeta^{\varepsilon_k}(u^{\varepsilon_k})\big)\,dx\\
&=\sum_{i=1}^N \frac{1}{p_i}\,\int_B |U_{x_i}|^{p_i}\, dx-\int_B f\,(\zeta(U)-\zeta(u))\,dx.
\end{split}
\end{equation}
Observe that for the convergence of the lower order term, we used that $f^{\varepsilon_k}$ converges strongly in \(L^1(B)\), that $U^{\varepsilon_k}$ and $u^{\varepsilon_k}$ are equibounded in $L^\infty(B)$ and converge almost everywhere  to $U$ and $u$ respectively and that $\zeta^{\varepsilon_k}$   converges uniformly to the Lipschitz function $\zeta$. This shows that 
\[
\sum_{i=1}^N\frac{1}{p_i}\, \int_B |u_{x_i}|^{p_i}\, dx -\int_{\Omega'} f\,\zeta(u)\,dx
\leq \sum_{i=1}^N\frac{1}{p_i}\, \int_B |U_{x_i}|^{p_i}\, dx -\int_{\Omega'} f\,\zeta(U)\,dx.
\]
The $L^\infty$-boundedness of $u^\varepsilon$ proved in Lemma \ref{lm:below}, gives that  \(\|u\|_{L^{\infty}(B)}\leq M\). A similar estimate holds for $U$  by assumption. Since \(\zeta(t)=t\) for every \(t\in [-M,M]\), one gets
\[
\mathfrak{F}_{\mathbf{p}}(u;B)\leq \mathfrak{F}_{\mathbf{p}}(U;B). 
\] 
By the strict convexity of the functional $\mathfrak{F}_\mathbf{p}$, the minimizer must be unique and thus we get $u=U$, as desired.
\par
In order to prove \eqref{troppoforte!}, we can adapt the argument of \cite[Lemma 2.2]{BB2}. By \eqref{maroc} we get
\begin{equation}
\label{norms}
\lim_{k\to \infty} \sum_{i=1}^N \frac{1}{p_i}\,\int_B \left|u^{\varepsilon_k}_{x_i}\right|^{p_i}\,dx=\sum_{i=1}^N \frac{1}{p_i}\,\int_B \left|U_{x_i}\right|^{p_i}\,dx.
\end{equation}
For every \(i=1,\dots,N\), we rely on the lower semicontinuity of the $L^{p_i}$ norm to get
\[
\liminf_{k\to \infty} \int_{B} \left| u^{\varepsilon_k}_{x_i}\right|^{p_i}\,dx \geq \int_{B} \left|U_{x_i}\right|^{p_i}\,dx, 
\]
In connection with \eqref{norms}, this implies that 
\[
\lim_{k\to \infty} \int_B \left|u^{\varepsilon_k}_{x_i}\right|^{p_i}\,dx=\int_B \left|U_{x_i}\right|^{p_i}\,dx.
\]
The convergence of the norms, in conjunction with the weak convergence, permits to infer that \((u^{\varepsilon_k}_{x_i})_{k\in \mathbb{N}}\) converges to \(U_{x_i}\) in \(L^{p_i}(B)\) for every \(i=1, \dots, N\) (see for example \cite[Theorem 2.11]{LL}).
\par
Moreover, since $\{u^{{\varepsilon}_k}\}_{k\in\mathbb{N}}$ is bounded by $M$ and converges almost everywhere in $B$ to $U$, the dominated convergence theorem implies that \((u^{\varepsilon_k})_{k\in \mathbb{N}}\) converges to \(U\) in \(L^{q}(B)\) for every \(1\leq q<\infty\). 
\par
Finally, we observe that we can repeat this argument with any subsequence of the original family $\{u^{\varepsilon}\}_{\varepsilon>0}$. Thus the above limit holds true for the whole family $\{u^{\varepsilon}\}_{0<\varepsilon\le \varepsilon_0}$ instead of $\{u^{\varepsilon_k}\}_{k\in \mathbb{N}}$ and \eqref{troppoforte!} follows.
\end{proof}

The following technical result is classical in the  Regularity Theory. This is taken from \cite[Lemma 6.1]{Gi} and we state it here for the reader's convenience.
\begin{lm}
\label{lm:giusti}
Let $0<r<R$ and let $Z:[r,R]\to [0,\infty)$ be a bounded function. Assume that for $r\le s<t\le R$ we have
\[
Z(s)\le \frac{\mathcal{A}}{(t-s)^{\alpha_0}}+\frac{\mathcal{B}}{(t-s)^{\beta_0}}+\mathcal{C}+\vartheta\,Z(t),
\]
with $\mathcal{A},\mathcal{B},\mathcal{C}\ge 0$, $\alpha_0\ge \beta_0>0$ and $0\le \vartheta<1$. Then we have
\[
Z(r)\le \left(\frac{1}{(1-\lambda)^{\alpha_0}}\,\frac{\lambda^{\alpha_0}}{\lambda^{\alpha_0}-\vartheta}\right)\,\left[\frac{\mathcal{A}}{(R-r)^{\alpha_0}}+\frac{\mathcal{B}}{(R-r)^{\beta_0}}+\mathcal{C}\right],
\]
where $\lambda$ is any number such that $\vartheta^\frac{1}{\alpha_0}<\lambda<1$.
\end{lm}

\section{Caccioppoli--type inequalities for the gradient}\label{caccio}
Throughout this section, we will assume that $1< p_1\le \dots \le p_N$, without any further restriction.
In what follows, we will use the following function
\begin{equation}
\label{Ggrande}
\mathcal{G}_\varepsilon(z)=\Big((G_\varepsilon(z)-1)_++1\Big),\qquad \mbox{ for every } z\in\mathbb{R}^N,
\end{equation}
where $G_\varepsilon$ is the same function as in \eqref{G}.
\begin{prop}[Caccioppoli inequality for power functions of the gradient]
\label{prop:caccioespilon2}
Let $1<p_1\le p_2\le \dots\le p_N<\infty$. For every $\alpha\ge 0$ and every non-negative $\eta\in C^2_0(B)$, we have
\begin{equation}
\label{caccioppola}
\begin{split}
\sum_{i=1}^N\int g''_{i,\varepsilon}(u_{x_i}^\varepsilon)&\,\left|\Big(\mathcal{G}_\varepsilon(\nabla u^\varepsilon)^\frac{\alpha+1}{2}\Big)_{x_i}\right|^2\,\eta^2\,dx\\&\le C\,(\alpha+1)^2\int \mathcal{G}_\varepsilon(\nabla u^\varepsilon)^{\alpha+2-\frac{2}{p_N}}\,\Big(|\nabla\eta|^2+\eta\,|D^2 \eta|\Big)\,dx\\
&+C\,(\alpha+1)^2\,\int |f^\varepsilon|^2\, \mathcal{G}_\varepsilon(\nabla u^\varepsilon)^\alpha\,\eta^2\,dx,
\end{split}
\end{equation}
for some $C=C(N,p_1,p_N)>0$.
\end{prop}

\begin{proof}
We are going to use a trick based on integration by parts, taken from \cite[Theorem 1]{Na}
 (see also \cite{FF}). This permits to circumvent the use of the upper bound on the Hessian of the function $G_\varepsilon$.
We start by fixing $k\in\{1,\dots,N\}$ and inserting in \eqref{Eulerdiff} the test function 
\[
\varphi=F(G_{\varepsilon}(\nabla u^\varepsilon))\,g'_{k,\varepsilon}(u^\varepsilon_{x_k})\,\eta^2,
\] 
where $F$ is a non-negative $C^1$ monotone non-decreasing function, that will be specified later on. This is a feasible test function, thanks to the regularity of $u^\varepsilon$.
Thus we get 
\begin{equation}
\label{1}
\begin{split}
\sum_{i=1}^N \int g_{i,\varepsilon}''(u^\varepsilon_{x_i})\,u^{\varepsilon}_{x_k\,x_i}&\Big(F(G_{\varepsilon}(\nabla u^\varepsilon))\Big)_{x_i}\,g'_{k,\varepsilon}(u^\varepsilon_{x_k})\,\eta^2\,dx\\
&+\sum_{i=1}^N\int g_{i,\varepsilon}''(u^\varepsilon_{x_i})\,u^{\varepsilon}_{x_k\,x_i}\,F(G_{\varepsilon}(\nabla u^\varepsilon))\,\Big(g'_{k,\varepsilon}(u^\varepsilon_{x_k})\Big)_{x_i}\,\eta^2\,dx\\
&=-2\,\sum_{i=1}^N\int g_{i,\varepsilon}''(u^\varepsilon_{x_i})\,u^{\varepsilon}_{x_k\,x_i}\,F(G_{\varepsilon}(\nabla u^\varepsilon))\,g'_{k,\varepsilon}(u^\varepsilon_{x_k})\,\eta\,\eta_{x_i}\,dx\\
&+\int (f^{\varepsilon}(\zeta^{\varepsilon})'(u^{\varepsilon}))_{x_k}\,F(G_{\varepsilon}(\nabla u^\varepsilon))\,g'_{k,\varepsilon}(u^\varepsilon_{x_k})\,\eta^2\,dx.
\end{split}
\end{equation}
We observe that 
\[
g_{i,\varepsilon}''(u^\varepsilon_{x_i})\,u^{\varepsilon}_{x_k\,x_i}=\Big(g_{i,\varepsilon}'(u^\varepsilon_{x_i})\Big)_{x_k}.
\]
Then by integrating by parts on the right-hand side\footnote{This is the trick in \cite{Na} (and \cite{FF}) mentioned above.} of \eqref{1}, we obtain
\begin{equation}
\label{2}
\begin{split}
\sum_{i=1}^N &\int g_{i,\varepsilon}''(u^\varepsilon_{x_i})\,u^{\varepsilon}_{x_k\,x_i}\Big(F(G_{\varepsilon}(\nabla u^\varepsilon))\Big)_{x_i}\,g'_{k,\varepsilon}(u^\varepsilon_{x_k})\,\eta^2\,dx\\
&+\sum_{i=1}^N\int g_{i,\varepsilon}''(u^\varepsilon_{x_i})\,u^{\varepsilon}_{x_k\,x_i}\,F(G_{\varepsilon}(\nabla u^\varepsilon))\,\Big(g'_{k,\varepsilon}(u^\varepsilon_{x_k})\Big)_{x_i}\,\eta^2\,dx\\
&=2\,\sum_{i=1}^N\int g_{i,\varepsilon}'(u^\varepsilon_{x_i})\,\Big(F(G_{\varepsilon}(\nabla u^\varepsilon))\Big)_{x_k}\,g'_{k,\varepsilon}(u^\varepsilon_{x_k})\,\eta\,\eta_{x_i}\,dx\\
&+2\,\sum_{i=1}^N\int g_{i,\varepsilon}'(u^\varepsilon_{x_i})\,F(G_{\varepsilon}(\nabla u^\varepsilon))\,\Big(g'_{k,\varepsilon}(u^\varepsilon_{x_k})\Big)_{x_k}\,\eta\,\eta_{x_i}\,dx\\
&+2\,\sum_{i=1}^N\int g_{i,\varepsilon}'(u^\varepsilon_{x_i})\,F(G_{\varepsilon}(\nabla u^\varepsilon))\,g'_{k,\varepsilon}(u^\varepsilon_{x_k})\,(\eta_{x_k}\,\eta_{x_i}+\eta\,\eta_{x_i\,x_k})\,dx\\
&+\int (f^{\varepsilon}(\zeta^{\varepsilon})'(u^{\varepsilon}))_{x_k}\,F(G_{\varepsilon}(\nabla u^\varepsilon))\,g'_{k,\varepsilon}(u^\varepsilon_{x_k})\,\eta^2\,dx.
\end{split}
\end{equation}
This is valid for every $k=1,\dots,N$, we then take the sum over $k$.

On the left-hand side, the first term then becomes
\[
\begin{split}
&\sum_{i,k=1}^N \int g_{i,\varepsilon}''(u^\varepsilon_{x_i})\,u^{\varepsilon}_{x_k\,x_i}\Big(F(G_{\varepsilon}(\nabla u^\varepsilon))\Big)_{x_i}\,g'_{k,\varepsilon}(u^\varepsilon_{x_k})\,\eta^2\,dx\\
&=\sum_{i,k=1}^N \int g_{i,\varepsilon}''(u^\varepsilon_{x_i})\,\Big(F(G_{\varepsilon}(\nabla u^\varepsilon))\Big)_{x_i}\,\Big(g_{k,\varepsilon}(u^\varepsilon_{x_k})\Big)_{x_i}\,\eta^2\,dx\\
&=\sum_{i=1}^N \int g_{i,\varepsilon}''(u^\varepsilon_{x_i})\,\Big(F(G_{\varepsilon}(\nabla u^\varepsilon))\Big)_{x_i}\,\left(\sum_{k=1}^Ng_{k,\varepsilon}(u^\varepsilon_{x_k})\right)_{x_i}\,\eta^2\,dx\\
&=\sum_{i=1}^N \int g_{i,\varepsilon}''(u^\varepsilon_{x_i})\,\Big(F(G_{\varepsilon}(\nabla u^\varepsilon))\Big)_{x_i}\,\Big(G_{\varepsilon}(\nabla u^\varepsilon)\Big)_{x_i}\,\eta^2\,dx\\
&=\sum_{i=1}^N \int g_{i,\varepsilon}''(u^\varepsilon_{x_i})\,F'(G_{\varepsilon}(\nabla u^\varepsilon))\,\left|\Big(G_{\varepsilon}(\nabla u^\varepsilon)\Big)_{x_i}\right|^2\,\eta^2\,dx.
\end{split}
\]
For the second term of the left-hand side in \eqref{2}, we observe that
\[
\begin{split}
\sum_{i,k=1}^N&\int g_{i,\varepsilon}''(u^\varepsilon_{x_i})\,u^{\varepsilon}_{x_k\,x_i}\,F(G_{\varepsilon}(\nabla u^\varepsilon))\,\Big(g'_{k,\varepsilon}(u^\varepsilon_{x_k})\Big)_{x_i}\,\eta^2\,dx\\
&=\sum_{i,k=1}^N\int g_{i,\varepsilon}''(u^\varepsilon_{x_i})\,|u^{\varepsilon}_{x_k\,x_i}|^2\,F(G_{\varepsilon}(\nabla u^\varepsilon))\,g''_{k,\varepsilon}(u^\varepsilon_{x_k})\,\eta^2\,dx,
\end{split}
\]
and this is non-negative, since each $g_{k,\varepsilon}$ is convex and $F\ge 0$.
We thus obtain
\begin{equation}
\label{3bis}
\begin{split}
\sum_{i=1}^N &\int g_{i,\varepsilon}''(u^\varepsilon_{x_i})\,F'(G_{\varepsilon}(\nabla u^\varepsilon))\,\left|\Big(G_{\varepsilon}(\nabla u^\varepsilon)\Big)_{x_i}\right|^2\,\eta^2\,dx\\
&+\sum_{i,k=1}^N\int g_{i,\varepsilon}''(u^\varepsilon_{x_i})\,|u^{\varepsilon}_{x_k\,x_i}|^2\,F(G_{\varepsilon}(\nabla u^\varepsilon))\,g''_{k,\varepsilon}(u^\varepsilon_{x_k})\,\eta^2\,dx\\
&=2\,\sum_{i,k=1}^N\int g_{i,\varepsilon}'(u^\varepsilon_{x_i})\,\Big(F(G_{\varepsilon}(\nabla u^\varepsilon))\Big)_{x_k}\,g'_{k,\varepsilon}(u^\varepsilon_{x_k})\,\eta\,\eta_{x_i}\,dx\\
&+2\,\sum_{i,k=1}^N\int g_{i,\varepsilon}'(u^\varepsilon_{x_i})\,F(G_{\varepsilon}(\nabla u^\varepsilon))\,\Big(g'_{k,\varepsilon}(u^\varepsilon_{x_k})\Big)_{x_k}\,\eta\,\eta_{x_i}\,dx\\
&+2\,\sum_{i,k=1}^N\int g_{i,\varepsilon}'(u^\varepsilon_{x_i})\,F(G_{\varepsilon}(\nabla u^\varepsilon))\,g'_{k,\varepsilon}(u^\varepsilon_{x_k})\,(\eta_{x_k}\,\eta_{x_i}+\eta\,\eta_{x_i\,x_k})\,dx\\
&+\sum_{k=1}^N\int (f^{\varepsilon}(\zeta^{\varepsilon})'(u^{\varepsilon}))_{x_k}\,F(G_{\varepsilon}(\nabla u^\varepsilon))\,g'_{k,\varepsilon}(u^\varepsilon_{x_k})\,\eta^2\,dx.
\end{split}
\end{equation}
By \eqref{nondiff}, the second  term of the right-hand side can be written as
\begin{equation}\label{eq903}
-2\,\sum_{i=1}^N\int g_{i,\varepsilon}'(u^\varepsilon_{x_i})\,F(G_{\varepsilon}(\nabla u^\varepsilon))\,f^\varepsilon (\zeta^\varepsilon)'(u^\varepsilon)\,\eta\,\eta_{x_i}\,dx.
\end{equation}
By an integration by parts and  \eqref{nondiff} again, the last term on the right-hand side of \eqref{3bis} is equal to
\[
\begin{split}
-\sum_{k=1}^N&\int f^\varepsilon (\zeta^\varepsilon)'(u^\varepsilon)\,\Big(F(G_{\varepsilon}(\nabla u^\varepsilon))\Big)_{x_k}\,g'_{k,\varepsilon}(u^\varepsilon_{x_k})\,\eta^2\,dx\\
&+\int |f^\varepsilon (\zeta^\varepsilon)'(u^\varepsilon)|^2\,F(G_{\varepsilon}(\nabla u^\varepsilon))\,\eta^2\,dx\\
&-2\,\sum_{k=1}^N \int f^\varepsilon (\zeta^\varepsilon)'(u^\varepsilon)\,F(G_{\varepsilon}(\nabla u^\varepsilon))\,g'_{k,\varepsilon}(u^\varepsilon_{x_k})\,\eta\,\eta_{x_k}\,dx.
\end{split}
\]
We observe that the third term in the above sum  is equal to the quantity in \eqref{eq903}.
Hence,  \eqref{3bis} is equivalent to
\begin{equation}
\label{4bis}
\begin{split}
\mathcal{I}_1+\mathcal{I}_2&:=\sum_{i=1}^N \int g_{i,\varepsilon}''(u^\varepsilon_{x_i})\,F'(G_{\varepsilon}(\nabla u^\varepsilon))\,\left|\Big(G_{\varepsilon}(\nabla u^\varepsilon)\Big)_{x_i}\right|^2\,\eta^2\,dx\\
&+\sum_{i,k=1}^N\int g_{i,\varepsilon}''(u^\varepsilon_{x_i})\,|u^{\varepsilon}_{x_k\,x_i}|^2\,F(G_{\varepsilon}(\nabla u^\varepsilon))\,g''_{k,\varepsilon}(u^\varepsilon_{x_k})\,\eta^2\,dx\\
&=2\,\sum_{i,k=1}^N\int g_{i,\varepsilon}'(u^\varepsilon_{x_i})\,\Big(F(G_{\varepsilon}(\nabla u^\varepsilon))\Big)_{x_k}\,g'_{k,\varepsilon}(u^\varepsilon_{x_k})\,\eta\,\eta_{x_i}\,dx\\
&-4\,\sum_{i=1}^N\int g_{i,\varepsilon}'(u^\varepsilon_{x_i})\,F(G_{\varepsilon}(\nabla u^\varepsilon))\,f^\varepsilon (\zeta^\varepsilon)'(u^\varepsilon)\,\eta\,\eta_{x_i}\,dx\\
&+2\,\sum_{i,k=1}^N\int g_{i,\varepsilon}'(u^\varepsilon_{x_i})\,F(G_{\varepsilon}(\nabla u^\varepsilon))\,g'_{k,\varepsilon}(u^\varepsilon_{x_k})\,(\eta_{x_k}\,\eta_{x_i}+\eta\,\eta_{x_i\,x_k})\,dx\\
&-\sum_{k=1}^N\int f^\varepsilon (\zeta^\varepsilon)'(u^\varepsilon)\,\Big(F(G_{\varepsilon}(\nabla u^\varepsilon))\Big)_{x_k}\,g'_{k,\varepsilon}(u^\varepsilon_{x_k})\,\eta^2\,dx\\
&+\int |f^\varepsilon (\zeta^\varepsilon)'(u^\varepsilon)|^2\,F(G_{\varepsilon}(\nabla u^\varepsilon))\,\eta^2\,dx=:2\,\mathcal{I}_3+4\,\mathcal{F}_1+2\,\mathcal{I}_4+\mathcal{F}_2+\mathcal{F}_3.
\end{split}
\end{equation}

We first estimate $\mathcal{I}_3$: by Young's inequality, we have for every $\tau>0$ 
\[
\begin{split}
|\mathcal{I}_3|&\le \int \sum_{k=1}^N \left(\sum_{i=1}^N|g_{i,\varepsilon}'(u^\varepsilon_{x_i})|\right)\,\left(F'(G_{\varepsilon}(\nabla u^\varepsilon))\,\left|\Big(G_{\varepsilon}(\nabla u^\varepsilon)\Big)_{x_k}g'_{k,\varepsilon}(u^\varepsilon_{x_k})\right|\right)\,\eta\,|\nabla\eta|\,dx\\
&\le \frac{1}{2\,\tau}\,\int \left(\sum_{i=1}^N|g_{i,\varepsilon}'(u^\varepsilon_{x_i})|\right)^2\,F'(G_{\varepsilon}(\nabla u^\varepsilon))\,\left(\sum_{k=1}^N \frac{\Big(g_{k,\varepsilon}'(u_{x_k}^\varepsilon)\Big)^2}{g''_{k,\varepsilon}(u^\varepsilon_{x_k})}\right)|\nabla \eta|^2\\
&+\frac{\tau}{2}\,\int \sum_{k=1}^N F'(G(\nabla u^\varepsilon))\,\left|\Big(G_{\varepsilon}(\nabla u^\varepsilon)\Big)_{x_k}\right|^2\,g''_{k,\varepsilon}(u^\varepsilon_{x_k})\,\eta^2\,dx\\
&\le \frac{1}{2\,\tau}\,\int \left(\sum_{i=1}^N|g_{i,\varepsilon}'(u^\varepsilon_{x_i})|\right)^2\,F'(G_{\varepsilon}(\nabla u^\varepsilon))\,\left(\sum_{k=1}^N \frac{\Big(g_{k,\varepsilon}'(u_{x_k}^\varepsilon)\Big)^2}{g''_{k,\varepsilon}(u^\varepsilon_{x_k})}\right)|\nabla \eta|^2\,dx +\frac{\tau}{2}\,\mathcal{I}_1.
\end{split}
\]
By taking $\tau=1/2$, we can absorb the term $\mathcal{I}_1$ on the right-hand side and obtain from \eqref{4bis}
\begin{equation}
\label{5}
\begin{split}
\frac{1}{2}\,\mathcal{I}_1+\mathcal{I}_2
\le 2\,&\int \left(\sum_{i=1}^N|g_{i,\varepsilon}'(u^\varepsilon_{x_i})|\right)^2\,F'(G_{\varepsilon}(\nabla u^\varepsilon))\,\left(\sum_{k=1}^N \frac{\Big(g_{k,\varepsilon}'(u_{x_k}^\varepsilon)\Big)^2}{g''_{k,\varepsilon}(u^\varepsilon_{x_k})}\right)|\nabla \eta|^2\,dx\\
&+4\,\mathcal{F}_1+2\,\mathcal{I}_4+\mathcal{F}_2+\mathcal{F}_3.
\end{split}
\end{equation}
The term 
\[
\mathcal{I}_4=\sum_{i,k=1}^N\int g_{i,\varepsilon}'(u^\varepsilon_{x_i})\,F(G_{\varepsilon}(\nabla u^\varepsilon))\,g'_{k,\varepsilon}(u^\varepsilon_{x_k})\,(\eta_{x_k}\,\eta_{x_i}+\eta\,\eta_{x_i\,x_k})\,dx,
\] 
is easier to handle: we simply have
\[
\begin{split}
|\mathcal{I}_4|\le \int \left(\sum_{i=1}^N|g_{i,\varepsilon}'(u^\varepsilon_{x_i})|\right)^2\,F(G_{\varepsilon}(\nabla u^\varepsilon))\,\Big(|\nabla\eta|^2+\eta\,|D^2\eta|\Big)\,dx.
\end{split}
\]
In conclusion, from \eqref{5} we get
\begin{equation}
\label{6}
\begin{split}
\mathcal{I}_1+2\,\mathcal{I}_2
&\le 4\,\int \left(\sum_{i=1}^N|g_{i,\varepsilon}'(u^\varepsilon_{x_i})|\right)^2\,F'(G_{\varepsilon}(\nabla u^\varepsilon))\,\left(\sum_{k=1}^N \frac{\Big(g_{k,\varepsilon}'(u_{x_k}^\varepsilon)\Big)^2}{g''_{k,\varepsilon}(u^\varepsilon_{x_k})}\right)|\nabla \eta|^2\,dx\\
&+8\,\mathcal{F}_1+4\,\int \left(\sum_{i=1}^N|g_{i,\varepsilon}'(u^\varepsilon_{x_i})|\right)^2\,F(G_{\varepsilon}(\nabla u^\varepsilon))\,\Big(|\nabla\eta|^2+\eta\,|D^2\eta|\Big)\,dx+2\,\mathcal{F}_2+2\,\mathcal{F}_3.
\end{split}
\end{equation}
We now treat the three terms containing $f^\varepsilon$: we start from 
\[
\begin{split}
\mathcal{F}_1&=-\sum_{i=1}^N\int g_{i,\varepsilon}'(u^\varepsilon_{x_i})\,F(G_{\varepsilon}(\nabla u^\varepsilon))\,f^\varepsilon (\zeta^\varepsilon)'(u^\varepsilon)\,\eta\,\eta_{x_i}\,dx\\
&\le \int \sum_{i=1}^N|g_{i,\varepsilon}'(u^\varepsilon_{x_i})|\,F(G_{\varepsilon}(\nabla u^\varepsilon))\,|f^\varepsilon (\zeta^\varepsilon)'(u^\varepsilon)|\,\eta\,|\nabla \eta|\,dx\\
&\le \frac{1}{2}\,\int \left(\sum_{i=1}^Ng_{i,\varepsilon}'(u^\varepsilon_{x_i})\right)^2\,F(G_{\varepsilon}(\nabla u^\varepsilon))\,|\nabla \eta|^2\,dx+\frac{1}{2}\,\int |f^\varepsilon(\zeta^\varepsilon)'(u^\varepsilon)|^2\,F(G_{\varepsilon}(\nabla u^\varepsilon))\,\eta^2\,dx.
\end{split}
\]
The last term coincides with \(\frac{1}{2}\mathcal{F}_3\) while the first term is bounded from above (up to a multiplicative constant) by the third term on the right-hand side of \eqref{6} .  Using also  that \(\|(\zeta^\varepsilon)'\|_{L^{\infty}(\mathbb{R})}\leq 1\), 
we thus get from \eqref{6}
\begin{equation}
\label{7}
\begin{split}
\mathcal{I}_1+2\,\mathcal{I}_2
&\le 4\,\int \left(\sum_{i=1}^N|g_{i,\varepsilon}'(u^\varepsilon_{x_i})|\right)^2\,F'(G_{\varepsilon}(\nabla u^\varepsilon))\,\left(\sum_{k=1}^N \frac{\Big(g_{k,\varepsilon}'(u_{x_k}^\varepsilon)\Big)^2}{g''_{k,\varepsilon}(u^\varepsilon_{x_k})}\right)|\nabla \eta|^2\,dx\\
&+8\,\int \left(\sum_{i=1}^N|g_{i,\varepsilon}'(u^\varepsilon_{x_i})|\right)^2\,F(G_{\varepsilon}(\nabla u^\varepsilon))\,\Big(|\nabla\eta|^2+\eta\,|D^2\eta|\Big)\,dx\\
&+6\,\int |f^\varepsilon|^2\,F(G_{\varepsilon}(\nabla u^\varepsilon))\,\eta^2\,dx+2\,\mathcal{F}_2.
\end{split}
\end{equation}

The last term \(\mathcal{F}_2\) contains second order derivatives of $u^\varepsilon$ that should be absorbed on the left-hand side. We proceed similarly as for $\mathcal{I}_3$ and estimate it as follows
\[
\begin{split}
\mathcal{F}_2
&=-\sum_{k=1}^N\int f^\varepsilon (\zeta^\varepsilon)'(u^\varepsilon)\,\Big(F(G_{\varepsilon}(\nabla u^\varepsilon))\Big)_{x_k}\,g'_{k,\varepsilon}(u^\varepsilon_{x_k})\,\eta^2\,dx\\
&\le \sum_{k=1}^N \int |f^\varepsilon|\,\left(F'(G_{\varepsilon}(\nabla u^\varepsilon))\,\left|\Big(G_{\varepsilon}(\nabla u^\varepsilon)\Big)_{x_k}g'_{k,\varepsilon}(u^\varepsilon_{x_k})\right|\right)\,\eta^2\,dx\\
&\le \frac{1}{2\,\tau}\,\sum_{k=1}^N\int |f^\varepsilon|^2\, F'(G_{\varepsilon}(\nabla u^\varepsilon))\,\frac{\Big(g'_{k,\varepsilon}(u^\varepsilon_{x_k})\Big)^2}{g''_{k,\varepsilon}(u^{\varepsilon}_{x_k})}\,\eta^2\,dx\\
&+\frac{\tau}{2}\, \sum_{k=1}^N \int F'(G_{\varepsilon}(\nabla u^\varepsilon))\,\left|\Big(G_{\varepsilon}(\nabla u^\varepsilon)\Big)_{x_k}\right|^2\,g''_{k,\varepsilon}(u^\varepsilon_{x_k})\,\eta^2\,dx\\
&= \frac{1}{2\,\tau}\,\sum_{k=1}^N\int |f^\varepsilon|^2\, F'(G_{\varepsilon}(\nabla u^\varepsilon))\,\frac{\Big(g'_{k,\varepsilon}(u^\varepsilon_{x_k})\Big)^2}{g''_{k,\varepsilon}(u^{\varepsilon}_{x_k})}\,\eta^2\,dx +\frac{\tau}{2} \mathcal{I}_1.
\end{split}
\]
thanks to Young's inequality. Here as always $\tau>0$ is arbitrary. By inserting this estimate in \eqref{7} and choosing $\tau=1/2$, we obtain 
\[
\begin{split}
\mathcal{I}_1+4\,\mathcal{I}_2
&\le 8\,\int \left(\sum_{i=1}^N|g_{i,\varepsilon}'(u^\varepsilon_{x_i})|\right)^2\,F'(G_{\varepsilon}(\nabla u^\varepsilon))\,\left(\sum_{k=1}^N \frac{\Big(g_{k,\varepsilon}'(u_{x_k}^\varepsilon)\Big)^2}{g''_{k,\varepsilon}(u^\varepsilon_{x_k})}\right)|\nabla \eta|^2\,dx\\
&+16\,\int \left(\sum_{i=1}^N|g_{i,\varepsilon}'(u^\varepsilon_{x_i})|\right)^2\,F(G_{\varepsilon}(\nabla u^\varepsilon))\,\Big(|\nabla\eta|^2+\eta\,|D^2\eta|\Big)\,dx\\
&+12\,\int |f^\varepsilon|^2\,F(G_{\varepsilon}(\nabla u^\varepsilon))\,\eta^2\,dx+4\,\sum_{k=1}^N\int |f^\varepsilon|^2\, F'(G_{\varepsilon}(\nabla u^\varepsilon))\,\frac{\Big(g'_{k,\varepsilon}(u^\varepsilon_{x_k})\Big)^2}{g''_{k,\varepsilon}(u^\varepsilon_{x_k})}\,\eta^2\,dx.
\end{split}
\]

We observe that if we set
\begin{equation}
\label{delti}
\delta_k=\max\left\{\frac{p_k}{p_k-1},p_k\right\}\qquad \mbox{ and }\qquad \overline\delta=\max_{k=1,\dots,N} \delta_k=\max\left\{\frac{p_1}{p_1-1},p_N\right\},
\end{equation}
by \eqref{g4} or \eqref{g4b}, we have
\begin{equation}
\label{gG}
\begin{split}
\sum_{k=1}^N\frac{\Big(g_{k,\varepsilon}'(u_{x_k}^\varepsilon)\Big)^2}{g''_{k,\varepsilon}(u^\varepsilon_{x_k})}\le \sum_{k=1}^N \delta_k\,g_{k,\varepsilon}(u^\varepsilon_{x_k})&\le 
\overline\delta\,G_{\varepsilon}(\nabla u^\varepsilon).
\end{split}
\end{equation}
Thus we have obtained
\begin{equation}
\label{10}
\begin{split}
\mathcal{I}_1+4\,\mathcal{I}_2
&\le 8\,\overline\delta\,\int \left(\sum_{i=1}^N|g_{i,\varepsilon}'(u^\varepsilon_{x_i})|\right)^2\,F'(G_{\varepsilon}(\nabla u^\varepsilon))\,G_{\varepsilon}(\nabla u^\varepsilon)\,|\nabla \eta|^2\,dx\\
&+16\,\int \left(\sum_{i=1}^N|g_{i,\varepsilon}'(u^\varepsilon_{x_i})|\right)^2\,F(G_{\varepsilon}(\nabla u^\varepsilon))\,\Big(|\nabla\eta|^2+\eta\,|D^2\eta|\Big)\,dx\\
&+12\,\int |f^\varepsilon|^2\,F(G_{\varepsilon}(\nabla u^\varepsilon))\,\eta^2\,dx+4\,\overline{\delta}\,\int |f^\varepsilon|^2\, F'(G_{\varepsilon}(\nabla u^\varepsilon))\,G_{\varepsilon}(\nabla u^\varepsilon)\,\eta^2\,dx.
\end{split}
\end{equation}
We now use Lemma \ref{lm:majorette} to estimate from above the right-hand side. Thus, from \eqref{10}, we get
\begin{equation}
\label{11}
\begin{split}
\mathcal{I}_1+4\mathcal{I}_2
&\le 8\,\overline\delta\,\left(p_N^\frac{p_N-1}{p_N}\right)^2\,\int \left(\sum_{i=1}^NG_{\varepsilon}(\nabla u^\varepsilon)^\frac{p_i-1}{p_i}\right)^2\,F'(G_{\varepsilon}(\nabla u^\varepsilon))\,G_{\varepsilon}(\nabla u^\varepsilon)\,|\nabla \eta|^2\,dx\\
&+16\,\left(p_N^\frac{p_N-1}{p_N}\right)^2\,\int \left(\sum_{i=1}^NG_{\varepsilon}(\nabla u^\varepsilon)^\frac{p_i-1}{p_i}\right)^2\,F(G_{\varepsilon}(\nabla u^\varepsilon))\,\Big(|\nabla\eta|^2+\eta\,|D^2\eta|\Big)\,dx\\
&+12\,\overline{\delta}\,\int |f^\varepsilon|^2\, \Big[F(G_{\varepsilon}(\nabla u^\varepsilon))+F'(G_{\varepsilon}(\nabla u^\varepsilon))\,G_{\varepsilon}(\nabla u^\varepsilon)\Big]\,\eta^2\,dx.
\end{split}
\end{equation}
By recalling the definition \eqref{Ggrande} of $\mathcal{G}_\varepsilon$,
we observe that $\mathcal{G}_\varepsilon(\nabla u^\varepsilon)\ge{G}_\varepsilon(\nabla u^\varepsilon)$ and $\mathcal{G}_\varepsilon(\nabla u^\varepsilon)\ge 1$. Thus in particular, we get
\[
\sum_{i=1}^NG_{\varepsilon}(\nabla u^\varepsilon)^\frac{p_i-1}{p_i}\le N\,\mathcal{G}_\varepsilon(\nabla u^\varepsilon)^\frac{p_N-1}{p_N},
\]
and from \eqref{11}, we obtain
\begin{equation}
\label{99}
\begin{split}
\mathcal{I}_1+4\mathcal{I}_2&\le 8\,\overline\delta\,\left(N\,p_N^\frac{p_N-1}{p_N}\right)^2\,\int \mathcal{G}_\varepsilon(\nabla u^\varepsilon)^{2\,\frac{p_N-1}{p_N}+1}\,F'(G_{\varepsilon}(\nabla u^\varepsilon))\,|\nabla \eta|^2\,dx\\
&+16\,\left(N\,p_N^\frac{p_N-1}{p_N}\right)^2\,\int \mathcal{G}_\varepsilon(\nabla u^\varepsilon)^{2\,\frac{p_N-1}{p_N}}\,F(G_{\varepsilon}(\nabla u^\varepsilon))\,\Big(|\nabla\eta|^2+\eta\,|D^2\eta|\Big)\,dx\\
&+12\,\overline{\delta}\,\int |f^\varepsilon|^2\, \Big[F(G_{\varepsilon}(\nabla u^\varepsilon))+F'(G_{\varepsilon}(\nabla u^\varepsilon))\,G_{\varepsilon}(\nabla u^\varepsilon)\Big]\,\eta^2\,dx.
\end{split}
\end{equation}

In order to conclude, we now make the choice
\[
F(t)=\big((t-1)_++1\big)^\alpha,\qquad \mbox{ with }\alpha>0.
\]
We observe that
\begin{equation}
\label{0030}
\left|\Big(\mathcal{G}_\varepsilon(\nabla u^\varepsilon)\Big)_{x_i}\right|^2\le \left|\Big(G_{\varepsilon}(\nabla u^\varepsilon)\Big)_{x_i}\right|^2,
\end{equation}
thanks to the definition of $\mathcal{G}_{\varepsilon}$. 

It follows that
\[
\begin{split}
F'(G_{\varepsilon}(\nabla u^\varepsilon))\,\left|\Big(G_\varepsilon(\nabla u^\varepsilon)\Big)_{x_i}\right|^2&\geq \alpha\,\mathcal{G}_\varepsilon(\nabla u^\varepsilon)^{\alpha-1}\,\left|\Big(\mathcal{G}_\varepsilon(\nabla u^\varepsilon)\Big)_{x_i}\right|^2=\frac{4\,\alpha}{(\alpha+1)^2}\,\left|\Big(\mathcal{G}_\varepsilon(\nabla u^\varepsilon)^\frac{\alpha+1}{2}\Big)_{x_i}\right|^2,
\end{split}
\]
and thus
\[
\begin{split}
\mathcal{I}_1&= \sum_{i=1}^N \int g_{i,\varepsilon}''(u^\varepsilon_{x_i})\,F'(G_{\varepsilon}(\nabla u^\varepsilon))\,\left|\Big(G_{\varepsilon}(\nabla u^\varepsilon)\Big)_{x_i}\right|^2\,\eta^2\,dx\\
&\geq \frac{4\,\alpha}{(\alpha+1)^2}\,\sum_{i=1}^N \int g_{i,\varepsilon}''(u^\varepsilon_{x_i})\,\left|\Big(\mathcal{G}_\varepsilon(\nabla u^\varepsilon)^\frac{\alpha+1}{2}\Big)_{x_i}\right|^2\,\eta^2\,dx.
\end{split}
\]
From \eqref{99} we get 
\begin{equation}
\label{12}
\begin{split}
\frac{4\,\alpha}{(\alpha+1)^2}\,\sum_{i=1}^N \int &g_{i,\varepsilon}''(u^\varepsilon_{x_i})\,\left|\Big(\mathcal{G}_\varepsilon(\nabla u^\varepsilon)^\frac{\alpha+1}{2}\Big)_{x_i}\right|^2\,\eta^2\,dx+4\,\mathcal{I}_2\\
&\le C\,(\alpha+1)\,\int \mathcal{G}_\varepsilon(\nabla u^\varepsilon)^{\alpha+2\,\frac{p_N-1}{p_N}}\,\Big(|\nabla\eta|^2+\eta\,|D^2\eta|\Big)\,dx\\
&+C\,(\alpha+1)\,\int |f^\varepsilon|^2\, \mathcal{G}_\varepsilon(\nabla u^\varepsilon)^\alpha\,\eta^2\,dx,
\end{split}
\end{equation}
for some $C=C(N,p_1,p_N)>0$. We are only left with estimating $\mathcal{I}_2$ from below: recall that we have 
\[
\begin{split}
\mathcal{I}_2&=\sum_{i,k=1}^N\int g_{i,\varepsilon}''(u^\varepsilon_{x_i})\,|u^{\varepsilon}_{x_k\,x_i}|^2\,F(G_{\varepsilon}(\nabla u^\varepsilon))\,g''_{k,\varepsilon}(u^\varepsilon_{x_k})\,\eta^2\,dx\\
&=\sum_{i,k=1}^N\int g_{i,\varepsilon}''(u^\varepsilon_{x_i})\,|u^{\varepsilon}_{x_k\,x_i}|^2\,\mathcal{G}_\varepsilon(\nabla u^\varepsilon)^\alpha\,g''_{k,\varepsilon}(u^\varepsilon_{x_k})\,\eta^2\,dx.
\end{split}
\]
We now observe that by \eqref{0030} and through some lengthy though elementary computations, we get
\[
\begin{split}
\left|\Big(\mathcal{G}_\varepsilon(\nabla u^\varepsilon)^\frac{1}{2}\Big)_{x_i}\right|^2=\frac{1}{4\,\mathcal{G}_\varepsilon(\nabla u^\varepsilon)}\,\left|\Big(\mathcal{G}_\varepsilon(\nabla u^\varepsilon\Big)_{x_i}\right|^2&\le \frac{1}{4\,\mathcal{G}_\varepsilon(\nabla u^\varepsilon)}\,\left|\Big(G_{\varepsilon}(\nabla u^\varepsilon\Big)_{x_i}\right|^2\\
&=\frac{1}{4\,\mathcal{G}_\varepsilon(\nabla u^\varepsilon)}\,\left|\left(\sum_{k=1}^N g_{k,\varepsilon}(u_{x_k}^\varepsilon)\right)_{x_i}\right|^2\\
&=\frac{1}{4\,\mathcal{G}_\varepsilon(\nabla u^\varepsilon)}\,\left|\sum_{k=1}^N g'_{k,\varepsilon}(u_{x_k}^\varepsilon)\,u^\varepsilon_{x_k\,x_i}\right|^2\\
&\le\frac{N}{4\,\mathcal{G}_\varepsilon(\nabla u^\varepsilon)}\,\sum_{k=1}^N |g'_{k,\varepsilon}(u_{x_k}^\varepsilon)|^2\,|u^\varepsilon_{x_k\,x_i}|^2. 
\end{split}
\]
We then apply \eqref{g4} or \eqref{g4b} on the last term, so to get
\[
\begin{split}
\left|\Big(\mathcal{G}_\varepsilon(\nabla u^\varepsilon)^\frac{1}{2}\Big)_{x_i}\right|^2&\le\frac{N}{4\,\mathcal{G}_\varepsilon(\nabla u^\varepsilon)}\,\sum_{k=1}^N \delta_k\,g_{k,\varepsilon}''(u^\varepsilon_{x_k})\,g_{k,\varepsilon}(u^\varepsilon_{x_k})\,|u^\varepsilon_{x_k\,x_i}|^2,
\end{split}
\]
where $\delta_k$ is the same quantity defined in \eqref{delti}.
We further observe that 
\[
\frac{g_{k,\varepsilon}(u^\varepsilon_{x_k})}{\mathcal{G}_\varepsilon(\nabla u^\varepsilon)}\le \frac{G_{\varepsilon}(\nabla u^\varepsilon)}{\mathcal{G}_\varepsilon(\nabla u^\varepsilon)}\le 1.
\]
This discussion leads us to 
\[
\frac{4}{N}\,\frac{1}{\overline\delta}\left|\Big(\mathcal{G}_\varepsilon(\nabla u^\varepsilon)^\frac{1}{2}\Big)_{x_i}\right|^2\le\sum_{k=1}^N g_{k,\varepsilon}''(u^\varepsilon_{x_k})\,|u^\varepsilon_{x_k\,x_i}|^2.
\]
By inserting this inequality in $\mathcal{I}_2$, we get
\[
\begin{split}
\mathcal{I}_2&\ge \frac{4}{N}\,\frac{1}{\overline\delta}\,\sum_{i=1}^N\int g_{i,\varepsilon}''(u^\varepsilon_{x_i})\,\mathcal{G}_\varepsilon(\nabla u^\varepsilon)^\alpha\,\left|\Big(\mathcal{G}_\varepsilon(\nabla u^\varepsilon)^\frac{1}{2}\Big)_{x_i}\right|^2\eta^2\,dx\\
&=\frac{4}{(\alpha+1)^2}\,\frac{1}{N\,\overline\delta}\,\sum_{i=1}^N \int g_{i,\varepsilon}''(u^\varepsilon_{x_i})\,\left|\Big(\mathcal{G}_\varepsilon(\nabla u^\varepsilon)^\frac{\alpha+1}{2}\Big)_{x_i}\right|^2\,\eta^2\,dx.
\end{split}
\]
Finally, we can use this estimate in \eqref{12}, so as to get the desired conclusion for $\alpha>0$. The limit case $\alpha=0$ can now be simply obtained by taking the limit $\alpha$ goes to $0$ in the previously obtained estimate, since the relevant constant remains bounded.
\end{proof}

\begin{prop}[Slow Moser's iteration]
\label{prop:slow}
Let $1<p_1\le p_2\le \dots\le p_N<\infty$.
For every $\vartheta\ge 2/p_N'$ and every non-negative function $\eta\in C^2_0(B)$, we have
\begin{equation}
\label{slowmoser}
\begin{split}
\int \mathcal{G}_\varepsilon(\nabla u^\varepsilon)^{\vartheta+\frac{2}{p_N}}\,\eta^2\,dx&\le \frac{C^\vartheta}{\vartheta}\,\int\eta^2\,dx+C\,\|u^\varepsilon\|_{L^\infty(B)}^2\,\vartheta^2\,\int \mathcal{G}_\varepsilon(\nabla u^\varepsilon)^\vartheta\,\Big(|\nabla \eta|^2+\eta\,|D^2 \eta|\Big)\,dx\\
&+C^\vartheta\,\vartheta^{\vartheta}\,\|u^\varepsilon\|_{L^\infty(B)}^{\vartheta+\frac{2}{p_N}}\,\int |f^\varepsilon|^{\vartheta+\frac{2}{p_N}}\,\eta^2\,dx,
\end{split}
\end{equation}
for some $C=C(N,p_1,p_N)>0$.
\end{prop}
\begin{proof}
We start by taking $\beta\ge 1$ and writing
\[
\begin{split}
\int \mathcal{G}_\varepsilon(\nabla u^\varepsilon)^{\beta+1}\,\eta^2\,dx&=\int \mathcal{G}_\varepsilon(\nabla u^\varepsilon)^{\beta}\,\big((G_{\varepsilon}(\nabla u^\varepsilon)-1)_++1\big)\,\eta^2\,dx\\
&\le \int \mathcal{G}_\varepsilon(\nabla u^\varepsilon)^{\beta}\,\eta^2\,dx+\int \mathcal{G}_\varepsilon(\nabla u^\varepsilon)^\beta\,\sum_{k=1}^N g_{k,\varepsilon}(u_{x_k}^\varepsilon)\,\eta^2\,dx.
\end{split}
\]
We observe that if we set
\[
\sigma_k=\max\left\{1,2^{\frac{p_k-2}{2}}\right\}\qquad \mbox{ and }\qquad \overline\sigma=\max_{k=1,\dots,N} \sigma_k=\max\left\{1,2^{\frac{p_N-2}{2}}\right\},
\]
using \eqref{g3} or \eqref{g3b}  on the second integral, we get
\[
\begin{split}
\int \mathcal{G}_\varepsilon(\nabla u^\varepsilon)^{\beta+1}\,\eta^2\,dx
&\le \int\mathcal{G}_\varepsilon(\nabla u^\varepsilon)^{\beta}\,\eta^2\,dx+\int \mathcal{G}_\varepsilon(\nabla u^\varepsilon)^\beta\,\sum_{k=1}^N \sigma_k\frac{\varepsilon^\frac{p_k}{2}}{p_k}\,\eta^2\,dx\\
&+\int \mathcal{G}_\varepsilon(\nabla u^\varepsilon)^\beta\,\sum_{k=1}^N \sigma_k\frac{g'_{k,\varepsilon}(u_{x_k}^\varepsilon)}{p_k}\,u_{x_k}^\varepsilon\,\eta^2\,dx.\\
\end{split}
\]
By recalling that $g_{k,\varepsilon}'(t)\,t\ge 0$ and using that $0<\varepsilon\le 1$, we get
\[
\begin{split}
\int \mathcal{G}_\varepsilon(\nabla u^\varepsilon)^{\beta+1}\,\eta^2\,dx&\le \int \mathcal{G}_\varepsilon(\nabla u^\varepsilon)^{\beta}\,\eta^2\,dx+\overline{\sigma}\,\frac{N}{p_1}\,\int \mathcal{G}_\varepsilon(\nabla u^\varepsilon)^\beta\,\eta^2\,dx\\
&+\frac{\overline\sigma}{p_1}\,\int \mathcal{G}_\varepsilon(\nabla u^\varepsilon)^\beta\,\sum_{k=1}^N g'_{k,\varepsilon}(u_{x_k}^\varepsilon)\,u_{x_k}^\varepsilon\,\eta^2\,dx.\\
\end{split}
\]
On the last term, using product rule and equation \eqref{nondiff}, we obtain
\[
\begin{split}
\int \mathcal{G}_\varepsilon(\nabla u^\varepsilon)^{\beta+1}\,\eta^2\,dx&\le \left(1+\overline\sigma\,\frac{N}{p_1}\right)\,\int\mathcal{G}_\varepsilon(\nabla u^\varepsilon)^{\beta}\,\eta^2\,dx-\frac{\overline\sigma}{p_1}\,\sum_{k=1}^N\int \Big(\mathcal{G}_\varepsilon(\nabla u^\varepsilon)^\beta\Big)_{x_k}\,g'_{k,\varepsilon}(u_{x_k}^\varepsilon)\,u^\varepsilon\,\eta^2\,dx\\
&-\frac{2\,\overline\sigma}{p_1}\,\sum_{k=1}^N\int \mathcal{G}_\varepsilon(\nabla u^\varepsilon)^\beta\,g'_{k,\varepsilon}(u_{x_k}^\varepsilon)\,u^\varepsilon\,\eta\,\eta_{x_k}\,dx\\
&+\frac{\overline\sigma}{p_1}\,\int \mathcal{G}_\varepsilon(\nabla u^\varepsilon)^\beta\,f^\varepsilon (\zeta^\varepsilon)'(u^\varepsilon)\,u^\varepsilon\,\eta^2\,dx.
\end{split}
\]
By using that $u^\varepsilon$ is bounded and that $0\le (\zeta^\varepsilon)'\le 1$, we get
\begin{equation}
\label{choe2}
\begin{split}
\int \mathcal{G}_\varepsilon(\nabla u^\varepsilon)^{\beta+1}\,\eta^2\,dx&\le \left(1+\overline\sigma\,\frac{N}{p_1}\right)\,\int\mathcal{G}_\varepsilon(\nabla u^\varepsilon)^{\beta}\,\eta^2\,dx\\
&+\frac{\overline\sigma\,\|u^\varepsilon\|_{L^\infty(B)}}{p_1}\,\sum_{k=1}^N\int \left|\Big(\mathcal{G}_\varepsilon(\nabla u^\varepsilon)^\beta\Big)_{x_k}\right|\,|g'_{k,\varepsilon}(u_{x_k}^\varepsilon)|\,\eta^2\,dx\\
&+\frac{2\,\overline\sigma\,\|u^\varepsilon\|_{L^\infty(B)}}{p_1}\,\sum_{k=1}^N\int \mathcal{G}_\varepsilon(\nabla u^\varepsilon)^\beta\,|g'_{k,\varepsilon}(u_{x_k}^\varepsilon)|\,\eta\,|\eta_{x_k}|\,dx\\
&+\frac{\overline\sigma\,\|u^\varepsilon\|_{L^\infty(B)}}{p_1}\,\int \mathcal{G}_\varepsilon(\nabla u^\varepsilon)^\beta\,|f^\varepsilon|\,\eta^2\,dx.
\end{split}
\end{equation}
Now, \eqref{majorette} together with the fact that $\mathcal{G}_{\varepsilon}\ge G_{\varepsilon}$ and $\mathcal{G}_{\varepsilon}\ge 1$ entail
\[
\begin{split}
&\frac{2\,\overline\sigma\,\|u^\varepsilon\|_{L^\infty(B)}}{p_1}\,\sum_{k=1}^N \int \mathcal{G}_\varepsilon(\nabla u^\varepsilon)^\beta\,|g'_{k,\varepsilon}(u_{x_k}^\varepsilon)|\,\eta\,|\eta_{x_k}|\,dx\\
&\le \frac{2\,\overline\sigma\,\|u^\varepsilon\|_{L^\infty(B)}}{p_1}\,N\,p_N^\frac{p_N-1}{p_N}\,\int \mathcal{G}_\varepsilon(\nabla u^\varepsilon)^{\beta+1-\frac{1}{p_N}}\,\eta\,|\nabla \eta|\,dx\\
&\le \tau\,\int \mathcal{G}_\varepsilon(\nabla u^\varepsilon)^{\beta+1}\,\eta^2\,dx+\frac{1}{\tau}\,\left(N\,p_N^\frac{p_N-1}{p_N}\right)^2\,\left(\frac{\overline\sigma\,\|u^\varepsilon\|_{L^\infty(B)}}{p_1}\right)^2\,\int \mathcal{G}_\varepsilon(\nabla u^\varepsilon)^{\beta+1-\frac{2}{p_N}}\,|\nabla \eta|^2\,dx,
\end{split}
\]
where in the last inequality we applied Young's inequality.
By choosing $\tau=1/2$, we can absorb the term containing $\mathcal{G}_\varepsilon(\nabla u^\varepsilon)^{\beta+1}$ and get from \eqref{choe2}
\begin{equation}
\label{choe3}
\begin{split}
\frac{1}{2}\,\int \mathcal{G}_\varepsilon(\nabla u^\varepsilon)^{\beta+1}\,\eta^2\,dx&\le \left(1+\frac{\overline\sigma\,N}{p_1}\right)\,\int\mathcal{G}_\varepsilon(\nabla u^\varepsilon)^{\beta}\,\eta^2\,dx\\
&+ \frac{\overline\sigma\,\|u^\varepsilon\|_{L^\infty(B)}}{p_1}\,\sum_{k=1}^N\int \left|\Big(\mathcal{G}_\varepsilon(\nabla u^\varepsilon)^\beta\Big)_{x_k}\right|\,|g'_{k,\varepsilon}(u_{x_k}^\varepsilon)|\,\eta^2\,dx\\
&+2\,\left(N\,p_N^\frac{p_N-1}{p_N}\right)^2\,\left(\overline\sigma\,\frac{\|u^\varepsilon\|_{L^\infty(B)}}{p_1}\right)^2\,\int \mathcal{G}_\varepsilon(\nabla u^\varepsilon)^{\beta+1-\frac{2}{p_N}}\,|\nabla \eta|^2\,dx\\
&+\frac{\overline\sigma\,\|u^\varepsilon\|_{L^\infty(B)}}{p_1}\,\int \mathcal{G}_\varepsilon(\nabla u^\varepsilon)^\beta\,|f^\varepsilon|\,\eta^2\,dx.
\end{split}
\end{equation}
For the second term on the right-hand side, we use  Young's inequality: for every $\tau>0$,
\[
\begin{split}
 \frac{\overline\sigma\,\|u^\varepsilon\|_{L^\infty(B)}}{p_1}\,\sum_{k=1}^N\int &\left|\Big(\mathcal{G}_\varepsilon(\nabla u^\varepsilon)^\beta\Big)_{x_k}\right|\,|g'_{k,\varepsilon}(u_{x_k}^\varepsilon)|\,\eta^2\,dx\\
 &=\frac{\overline\sigma\,\|u^\varepsilon\|_{L^\infty(B)}}{p_1}\,\sum_{k=1}^N \beta\,\int\mathcal{G}_\varepsilon(\nabla u^\varepsilon)^{\beta-1}\,\left|\Big(\mathcal{G}_\varepsilon(\nabla u^\varepsilon)\Big)_{x_k}\right|\,|g'_{k,\varepsilon}(u_{x_k}^\varepsilon)|\,\eta^2\,dx\\
&\le \left(\frac{\overline\sigma\,\|u^\varepsilon\|_{L^\infty(B)}}{p_1}\right)^2\,\frac{\beta^2}{2\,\tau}\,\sum_{k=1}^N\int \mathcal{G}_\varepsilon(\nabla u^\varepsilon)^{\beta-2}\,\left|\Big(\mathcal{G}_\varepsilon(\nabla u^\varepsilon)\Big)_{x_k}\right|^2\,g''_{k,\varepsilon}(u_{x_k}^\varepsilon)\,\eta^2\,dx\\
&+\frac{\tau}{2}\,\sum_{k=1}^N\int\mathcal{G}_\varepsilon(\nabla u^\varepsilon)^\beta\,\frac{\Big(g_{k,\varepsilon}'(u_{x_k}^\varepsilon)\Big)^2}{g''_{k,\varepsilon}(u^\varepsilon_{x_k})}\,\eta^2\,dx\\
&=\frac{2}{\tau}\,\left(\frac{\overline\sigma\,\|u^\varepsilon\|_{L^\infty(B)}}{p_1}\right)^2\,\sum_{k=1}^N\int\left|\Big(\mathcal{G}_\varepsilon(\nabla u^\varepsilon)^\frac{\beta}{2}\Big)_{x_k}\right|^2\,g''_{k,\varepsilon}(u_{x_k}^\varepsilon)\,\eta^2\,dx\\
&+\frac{\tau}{2}\,\sum_{k=1}^N\int\mathcal{G}_\varepsilon(\nabla u^\varepsilon)^\beta\,\frac{\Big(g_{k,\varepsilon}'(u_{x_k}^\varepsilon)\Big)^2}{g''_{k,\varepsilon}(u^\varepsilon_{x_k})}\,\eta^2\,dx.
\end{split}
\]
We also notice that by \eqref{gG} we have
\[
\sum_{k=1}^N\frac{\Big(g_{k,\varepsilon}'(u_{x_k}^\varepsilon)\Big)^2}{g''_{k,\varepsilon}(u^\varepsilon_{x_k})}\le \overline\delta\,\mathcal{G}_\varepsilon(\nabla u^\varepsilon).
\]
Thus from \eqref{choe3}, we get
\[
\begin{split}
\frac{1}{2}\,\int \mathcal{G}_\varepsilon(\nabla u^\varepsilon)^{\beta+1}\,\eta^2\,dx&\le \left(1+\frac{\overline\sigma\,N}{p_1}\right)\,\int\mathcal{G}_\varepsilon(\nabla u^\varepsilon)^{\beta}\,\eta^2\,dx\\
&+\frac{2}{\tau}\,\left(\frac{\overline\sigma\,\|u^\varepsilon\|_{L^\infty(B)}}{p_1}\right)^2\,\sum_{k=1}^N\int\left|\Big(\mathcal{G}_\varepsilon(\nabla u^\varepsilon)^\frac{\beta}{2}\Big)_{x_k}\right|^2\,g''_{k,\varepsilon}(u_{x_k}^\varepsilon)\,\eta^2\,dx\\
&+\frac{\tau\,\overline\delta}{2}\int\mathcal{G}_\varepsilon(\nabla u^\varepsilon)^{\beta+1}\,\eta^2\\
&+2\,\left(N\,p_N^\frac{p_N-1}{p_N}\right)^2\,\left(\frac{\overline\sigma\,\|u^\varepsilon\|_{L^\infty(B)}}{p_1}\right)^2\,\int \mathcal{G}_\varepsilon(\nabla u^\varepsilon)^{\beta+1-\frac{2}{p_N}}\,|\nabla \eta|^2\,dx\\
&+\frac{\overline\sigma\,\|u^\varepsilon\|_{L^\infty(B)}}{p_1}\,\int \mathcal{G}_\varepsilon(\nabla u^\varepsilon)^\beta\,|f^\varepsilon|\,\eta^2\,dx.
\end{split}
\]
By choosing $\tau=1/(2\,\overline\delta)$, we can absorb again the term containing $\mathcal{G}_\varepsilon(\nabla u^\varepsilon)^{\beta+1}$ on the right-hand side and obtain
\[
\begin{split}
\frac{1}{4}\,\int \mathcal{G}_\varepsilon(\nabla u^\varepsilon)^{\beta+1}\,\eta^2\,dx&\le \left(1+\frac{\overline\sigma\,N}{p_1}\right)\,\int\mathcal{G}_\varepsilon(\nabla u^\varepsilon)^{\beta}\,\eta^2\,dx\\
&+4\,\overline\delta\,\left(\frac{\overline\sigma\,\|u^\varepsilon\|_{L^\infty(B)}}{p_1}\right)^2\,\sum_{k=1}^N\int\left|\Big(\mathcal{G}_\varepsilon(\nabla u^\varepsilon)^\frac{\beta}{2}\Big)_{x_k}\right|^2\,g''_{k,\varepsilon}(u_{x_k}^\varepsilon)\,\eta^2\,dx\\
&+2\,\left(N\,p_N^\frac{p_N-1}{p_N}\right)^2\,\left(\frac{\overline\sigma\,\|u^\varepsilon\|_{L^\infty(B)}}{p_1}\right)^2\,\int \mathcal{G}_\varepsilon(\nabla u^\varepsilon)^{\beta+1-\frac{2}{p_N}}\,|\nabla \eta|^2\,dx\\
&+\frac{\overline\sigma\,\|u^\varepsilon\|_{L^\infty(B)}}{p_1}\,\int \mathcal{G}_\varepsilon(\nabla u^\varepsilon)^\beta\,|f^\varepsilon|\,\eta^2\,dx.
\end{split}
\]
On the right-hand side, we now use the Caccioppoli inequality \eqref{caccioppola} with $\alpha=\beta-1\ge 0$, so to get
\[
\begin{split}
\sum_{k=1}^N\int g''_{k,\varepsilon}(u_{x_k}^\varepsilon)\,\left|\Big(\mathcal{G}_\varepsilon(\nabla u^\varepsilon)^\frac{\beta}{2}\Big)_{x_k}\right|^2\,\eta^2\,dx&\le C\,\beta^2\,\int \mathcal{G}_\varepsilon(\nabla u^\varepsilon)^{\beta+1-\frac{2}{p_N}}\,\Big(|\nabla\eta|^2+\eta\,|D^2 \eta|\Big)\,dx\\
&+C\,\beta^2\,\int |f^\varepsilon|^2\, \mathcal{G}_\varepsilon(\nabla u^\varepsilon)^{\beta-1}\,\eta^2\,dx,\\
\end{split}
\]
for some $C=C(N,p_1,p_N)>0$. 
This finally gives
\[
\begin{split}
\int \mathcal{G}_\varepsilon(\nabla u^\varepsilon)^{\beta+1}\,\eta^2\,dx&\le C\,\int\mathcal{G}_\varepsilon(\nabla u^\varepsilon)^{\beta}\,\eta^2\,dx\\
&+C\,\|u^\varepsilon\|_{L^\infty(B)}^2\,\beta^2\,\int \mathcal{G}_\varepsilon(\nabla u^\varepsilon)^{\beta+1-\frac{2}{p_N}}\,\Big(|\nabla \eta|^2+\eta\,|D^2 \eta|\Big)\,dx\\
&+C\,\|u^\varepsilon\|_{L^\infty(B)}^2\,\beta^2\,\int |f^\varepsilon|^2\, \mathcal{G}_\varepsilon(\nabla u^\varepsilon)^{\beta-1}\,\eta^2\,dx\\
&+C\,\|u^\varepsilon\|_{L^\infty(B)}\,\int \mathcal{G}_\varepsilon(\nabla u^\varepsilon)^\beta\,|f^\varepsilon|\,\eta^2\,dx,
\end{split}
\]
for some $C=C(N,p_1,p_N)>0$. On the first term on the right-hand side, we can use Young's inequality
\[
C\,\int\mathcal{G}_\varepsilon(\nabla u^\varepsilon)^{\beta}\,\eta^2\,dx\le \frac{\tau\,\beta}{\beta+1}\,\int \mathcal{G}_\varepsilon(\nabla u^\varepsilon)^{\beta+1}\,\eta^2\,dx+\frac{C^{\beta+1}}{\tau^\beta\,(\beta+1)}\,\int \eta^2\,dx.
\]
By choosing $\tau=1/2$, we can re-absorb the term $\mathcal{G}_\varepsilon(\nabla u^\varepsilon)^{\beta+1}$. This gives 
\[
\begin{split}
\frac{1}{2}\,\int \mathcal{G}_\varepsilon(\nabla u^\varepsilon)^{\beta+1}\,\eta^2\,dx&\le \frac{2^\beta\,C^{\beta+1}}{\beta+1}\,\int \eta^2\,dx\\
&+C\,\|u^\varepsilon\|_{L^\infty(B)}^2\,\beta^2\,\int \mathcal{G}_\varepsilon(\nabla u^\varepsilon)^{\beta+1-\frac{2}{p_N}}\,\Big(|\nabla \eta|^2+\eta\,|D^2 \eta|\Big)\,dx\\
&+C\,\|u^\varepsilon\|_{L^\infty(B)}^2\,\beta^2\,\int |f^\varepsilon|^2\, \mathcal{G}_\varepsilon(\nabla u^\varepsilon)^{\beta-1}\,\eta^2\,dx\\
&+C\,\|u^\varepsilon\|_{L^\infty(B)}\,\int \mathcal{G}_\varepsilon(\nabla u^\varepsilon)^\beta\,|f^\varepsilon|\,\eta^2\,dx.
\end{split}
\]
We proceed in a similar way for the two terms containing $f^{\varepsilon}$. By using Young's inequality with exponents
$\displaystyle{\frac{\beta+1}{\beta-1},\ \frac{\beta+1}{2}}$
and $\displaystyle{\frac{\beta+1}{\beta},\  \beta+1}$, respectively, 
we get
\[
\begin{split}
C\,\|u^\varepsilon\|_{L^\infty(B)}^2\,\beta^2\,\int |f^\varepsilon|^2\, \mathcal{G}_\varepsilon(\nabla u^\varepsilon)^{\beta-1}\,\eta^2\,dx
&+C\,\|u^\varepsilon\|_{L^\infty(B)}\,\int \mathcal{G}_\varepsilon(\nabla u^\varepsilon)^\beta\,|f^\varepsilon|\,\eta^2\,dx\\
&\le \tau\,\frac{\beta-1}{\beta+1}\,\int \mathcal{G}_\varepsilon(\nabla u^\varepsilon)^{\beta+1}\,\eta^2\,dx\\
&+\frac{2\,\tau^\frac{1-\beta}{2}\,C^\frac{\beta+1}{2}}{\beta+1}\,\|u^\varepsilon\|_{L^\infty(B)}^{\beta+1}\,\beta^{\beta+1}\,\int |f^\varepsilon|^{\beta+1}\eta^2\,dx\\
&+\tau\,\frac{\beta}{\beta+1}\,\int \mathcal{G}_\varepsilon(\nabla u^\varepsilon)^{\beta+1}\,\eta^2\,dx\\
&+\frac{C^{\beta+1}}{\tau^\beta\,(\beta+1)}\,\|u^\varepsilon\|_{L^\infty(B)}^{\beta+1}\,\int |f^\varepsilon|^{\beta+1}\,\eta^2\,dx.
\end{split}
\]
By choosing $\tau=1/8$, we can absorb again the terms containing the power $\beta+1$ of $\mathcal{G}_\varepsilon(\nabla u^\varepsilon)$.
This finally leads to the estimate \eqref{slowmoser}, up to renaming $
\vartheta=\beta+1-2/p_N$. 
This concludes the proof.
\end{proof}

\section{Uniform higher integrability}\label{sec:uniform}
In this section,  we establish a higher integrability estimate for $\nabla u^\varepsilon$, which will eventually lead to the result of Proposition \ref{prop:high}. We assume throughout the section  that $1< p_1\le \dots \le p_N$, without any further restriction.
\begin{prop}
\label{prop:plc}
Let $1<p_1\le p_2\le \dots\le p_N<\infty$. Then for every $B_r\Subset B_R\Subset B$ and every $\gamma\ge 2$,
we have
\[
\begin{split}
\int_{B_r} \mathcal{G}_\varepsilon(\nabla u^\varepsilon)^\gamma\,dx&\le \Gamma_1+\Gamma_2\,\int_{B_R} \mathcal{G}_\varepsilon(\nabla u^\varepsilon)\,dx,
\end{split}
\]
for two constants $\Gamma_1,\Gamma_2>0$ depending on 
\[
N,\, p_N,\, p_1,\,\gamma,\,R,\,R-r,\,\|f^\varepsilon\|_{L^\gamma(B_R)} \mbox{ and } \|u^\varepsilon\|_{L^\infty(B)}.
\]
\end{prop}
\begin{proof}
We take $\gamma\ge 2$ and define the sequence of exponents
\[
\vartheta_0=\frac{2}{p_N'},\qquad \vartheta_{i+1}=\vartheta_i+\frac{2}{p_N}=\frac{2}{p_N'}+(i+1)\,\frac{2}{p_N}=2+i\,\frac{2}{p_N},\ \mbox{ for } i\in \mathbb{N}.
\]
We set 
\begin{equation}
\label{i0}
i_0=\max\left\{i\in\mathbb{N}\, :\, i\le \frac{p_N}{2}\,\left(\gamma-2\right)\right\}.
\end{equation}
This in particular implies that 
\[
\vartheta_{i_0+1}\le \gamma<\vartheta_{i_0+2}.
\]
We now need to distinguish various cases, according to the values of $p_N$ and $\gamma$.
\vskip.2cm\noindent
{\bf Case A.1.} Here we assume that 
\begin{equation}
\label{casoa1}
p_N\le 2\qquad \mbox{ and }\qquad \frac{p_N}{2}\,(\gamma-2)\in\mathbb{N}.
\end{equation}
This is the simplest case: we get the estimate by iterating Proposition \ref{prop:slow} with exponents $\vartheta=\vartheta_i$ and  a suitable sequence of shrinking balls. 
\par
More precisely, we fix $B_r$ and $B_R$ as in the statement and
define the sequence of decreasing radii
\[
r_i=R-i\,\frac{R-r}{i_0+1},\qquad \mbox{ for } i=0,\dots,i_0+1.
\]
Accordingly, we take a cut-off function $\eta_i\in C^2_0(B_{r_i})$ for $i=0,\dots,i_0$, such that 
\[
0\le \eta_i\le 1,\qquad \eta_i \equiv 1 \mbox{ on } B_{r_{i+1}}, \qquad |\nabla \eta_i|^2+|D^2 \eta_i|\le \frac{C\,(i_0+1)^2}{(R-r)^2}.
\]
By applying \eqref{slowmoser} with $\vartheta=\vartheta_i$, $\eta=\eta_i$ and using the properties of the cut-off function, we get
\[
\begin{split}
\int_{B_{r_{i+1}}} \mathcal{G}_\varepsilon(\nabla u^\varepsilon)^{\vartheta_{i+1}}\,dx&\le \frac{C^{\vartheta_i}}{\vartheta_i}\,|B_R|+\frac{C\,\vartheta_i^2\,(i_0+1)^2}{(R-r)^2}\,\|u^\varepsilon\|^2_{L^\infty(B)}\,\int_{B_{r_i}} \mathcal{G}_\varepsilon(\nabla u^\varepsilon)^{\vartheta_i}\,dx\\
&+C^{\vartheta_i}\,\vartheta_i^{\vartheta_i}\,\|u^\varepsilon\|_{L^\infty(B)}^{\vartheta_{i+1}}\,\int_{B_{r_i}} |f^\varepsilon|^{\vartheta_{i+1}}\,dx,
\end{split}
\]
for a constant $C=C(N,p_1,p_N)>0$. By using H\"older's inequality on the right-hand side, we also get
\begin{equation}
\label{slowmoseri}
\begin{split}
\int_{B_{r_{i+1}}} \mathcal{G}_\varepsilon(\nabla u^\varepsilon)^{\vartheta_{i+1}}\,dx
&\le \frac{C^{\vartheta_i}}{\vartheta_i}\,|B_R|+\frac{C\,\vartheta_i^2\,(i_0+1)^2}{(R-r)^2}\,\|u^\varepsilon\|^2_{L^\infty(B)}\,\int_{B_{r_i}} \mathcal{G}_\varepsilon(\nabla u^\varepsilon)^{\vartheta_i}\,dx\\
&+C^{\vartheta_i}\,\vartheta_i^{\vartheta_i}\,|B_R|^{1-\frac{\vartheta_{i+1}}{\gamma}}\,\|u^\varepsilon\|_{L^\infty(B)}^{\vartheta_{i+1}}\,\left(\int_{B_{R}} |f^\varepsilon|^\gamma\,dx\right)^\frac{\vartheta_{i+1}}{\gamma}.
\end{split}
\end{equation}
Starting from $i=0$ and iterating \eqref{slowmoseri} from $0$ to $i_0$, we get
\begin{equation}
\label{mother}
\begin{split}
\int_{B_{r}} \mathcal{G}_\varepsilon(\nabla u^\varepsilon)^{\vartheta_{i_0+1}}\,dx&\le M+D_0\,\int_{B_R} \mathcal{G}_\varepsilon(\nabla u^\varepsilon)^\frac{2}{p_N'}\,dx.
\end{split}
\end{equation}
where we set for notational simplicity for every natural number $0\le k\le i_0$
\[
D_{k}=\left[\prod_{i=k}^{i_0}\frac{C\,\vartheta_i^2\,(i_0+1)^2}{(R-r)^2}\,\|u^\varepsilon\|^2_{L^\infty(B)}\right]=\left[\frac{C\,(i_0+1)^2}{(R-r)^2}\,\|u^\varepsilon\|^2_{L^\infty(B)}\right]^{i_0-k+1}\,\prod_{i=k}^{i_0} \vartheta_i^2,
\]
while
\[
M=\sum_{i=0}^{i_0}\left(\frac{C^{\vartheta_i}}{\vartheta_i}\,|B_R|+C^{\vartheta_i}\,\vartheta_i^{\vartheta_i}\,|B_R|^{1-\frac{\vartheta_{i+1}}{\gamma}}\,\left(\|u^\varepsilon\|_{L^\infty(B)}\,\|f\|_{L^\gamma(B_R)}\right)^{\vartheta_{i+1}}\right)\,D_{i+1},
\]
with the notational agreement that $D_{i_0+1}=1$. Their precise expression is not very important, but
we point out that $D_0$ and $M$ depend only on
\[
N,p_N,p_1,\gamma,R,R-r,\|f_\varepsilon\|_{L^\gamma(B_R)}\qquad \mbox{ and }\qquad \|u_\varepsilon\|_{L^\infty(B)}.
\]
Thanks to the assumption \eqref{casoa1} and recalling that $\mathcal{G}_{\varepsilon}\ge 1$, we get 
\[
\vartheta_{i_0+1}=\gamma\qquad \mbox{ and }\qquad  \int_{B_R} \mathcal{G}_\varepsilon(\nabla u^\varepsilon)^\frac{2}{p_N'}\,dx\le \int_{B_R} \mathcal{G}_\varepsilon(\nabla u^\varepsilon)\,dx,
\]
thus the desired conclusion follows from \eqref{mother}.
\vskip.2cm\noindent
{\bf Case A.2}. Here we assume that 
\[
p_N\le 2\qquad \mbox{ and }\qquad \frac{p_N}{2}\,(\gamma-2)\not\in\mathbb{N}.
\]
In light of the second assumption, we have
\[
\vartheta_{i_0+1}<\gamma<\vartheta_{i_0+2},
\]
where $i_0$ is the index defined in \eqref{i0}.
Thus in this case, we need an extra step of the iteration, by suitably adapting the choice of the exponent $\vartheta$. We take $B_r\Subset B_\varrho\Subset B_R$ where 
\[
\varrho=\frac{R+r}{2},
\]
and define the sequence of decreasing radii
\[
r_i=R-i\,\frac{R-\varrho}{i_0+1},\qquad \mbox{ for } i=0,\dots,i_0+1.
\]
We take a cut-off function $\eta_i\in C^2_0(B_{r_i})$ for $i=0,\dots,i_0$, such that 
\[
0\le \eta_i\le 1,\qquad \eta_i \equiv 1 \mbox{ on } B_{r_{i+1}}, \qquad |\nabla \eta_i|^2+|D^2 \eta_i|\le \frac{C\,(i_0+1)^2}{(R-\rho)^2}.
\]
By proceeding as above, we now get
 \begin{equation}
\label{mother2}
\begin{split}
\int_{B_{\varrho}} \mathcal{G}_\varepsilon(\nabla u^\varepsilon)^{\vartheta_{i_0+1}}\,dx&\le M+D_0\,\int_{B_R} \mathcal{G}_\varepsilon(\nabla u^\varepsilon)^\frac{2}{p_N'}\,dx.
\end{split}
\end{equation}
The last term can be estimated from above, by using again that $2/p_N'\le 1$. However,
in order to reach the desired exponent $\gamma$, we need to apply \eqref{slowmoser} once more. We take a cut-off function $\eta\in C^2_0(B_\varrho)$ such that 
\[
0\le \eta\le 1,\qquad \eta \equiv 1 \mbox{ on } B_{r}, \qquad |\nabla \eta|^2+|D^2 \eta|\le \frac{C\,(i_0+1)^2}{(\varrho-r)^2}.
\]
By applying \eqref{slowmoser} with\footnote{Observe that such a choice is feasible, since 
\[
\gamma-\frac{2}{p_N}\ge \frac{2}{p_N'}\qquad \Longleftrightarrow \qquad \gamma\ge 2.
\]} 
\[
\vartheta=\gamma-\frac{2}{p_N},
\]
and the cut-off function above, we get
\[
\begin{split}
\int_{B_r} \mathcal{G}_\varepsilon(\nabla u^\varepsilon)^\gamma\,dx&\le C\,|B_R|^N+C\,\|u^\varepsilon\|_{L^\infty(B)}^2\,\int_{B_\varrho} \mathcal{G}_\varepsilon(\nabla u^\varepsilon)^{\gamma-\frac{2}{p_N}}\,dx\\
&+C\,\|u^\varepsilon\|_{L^\infty(B)}^{\gamma}\,\int_{B_R} |f^\varepsilon|^\gamma\,dx.
\end{split}
\]
On the right hand side, the term containing $\mathcal{G}_\varepsilon(\nabla u^\varepsilon)$ is under control, since by construction
\[
\gamma-\frac{2}{p_N}<\vartheta_{i_0+2}-\frac{2}{p_N}=\vartheta_{i_0+1}.
\]
Thus we have 
\[
\int_{B_\varrho} \mathcal{G}_\varepsilon(\nabla u^\varepsilon)^{\gamma-\frac{2}{p_N}}\,dx\le \int_{B_\varrho} \mathcal{G}_\varepsilon(\nabla u^\varepsilon)^{\vartheta_{i_0+1}}\,dx,
\]
and the last term can be estimated by \eqref{mother2}.
\vskip.2cm\noindent
{\bf Case B}. Here we assume that $p_N>2$. The proof goes exactly as before, so, for every $r<R$ with $B_R\Subset B$ we certainly have
\[
\begin{split}
\int_{B_r} \mathcal{G}_\varepsilon(\nabla u^\varepsilon)^\gamma\,dx&\le \Gamma_1+\Gamma_2\,\int_{B_{\frac{R+r}{2}}} \mathcal{G}_\varepsilon(\nabla u^\varepsilon)^\frac{2}{p_N'}\,dx,
\end{split}
\]
 but now the major difference is that we need to estimate the integral on the right hand side.
Indeed, in this case $2/p_N'>1$ and we can not directly assure that this term is bounded, uniformly in $\varepsilon$.
We need to use an interpolation trick to get a reverse $L^{2/p_N'}-L^1$ estimate on $\mathcal{G}_\varepsilon(\nabla u^\varepsilon)$. We denote $\varrho=(R+r)/2$ and we observe that 
\[
1<\frac{2}{p_N'}<2,
\]
thus by interpolation in Lebesgue spaces, we get for every $\varrho\le s<t\le R$
\begin{equation}
\label{chiara}
\int_{B_s} \mathcal{G}_\varepsilon(\nabla u^\varepsilon)^\frac{2}{p_N'}\,dx\le \left(\int_{B_s} \mathcal{G}_\varepsilon(\nabla u^\varepsilon)\,dx\right)^\frac{2}{p_N}\,\left(\int_{B_s} \mathcal{G}_\varepsilon(\nabla u^\varepsilon)^2\,dx\right)^\frac{p_N-2}{p_N}.
\end{equation}
The $L^2$ norm on the right-hand side can in turn be estimated by means of \eqref{slowmoser}, observing that
\[
2=\frac{2}{p_N'}+\frac{2}{p_N}.
\]
By taking $\vartheta=2/p_N'$ and a cut-off function $\eta\in C^2_0(B_t)$ such that 
\[
0\le \eta\le 1,\qquad \eta \equiv 1 \mbox{ on } B_s, \qquad |\nabla \eta|^2+|D^2 \eta|\le \frac{C}{(t-s)^2},
\]
we thus get from \eqref{slowmoser}
\[
\begin{split}
\int_{B_s} \mathcal{G}_\varepsilon(\nabla u^\varepsilon)^2\,dx&\le C\,|B_R|+\frac{C}{(t-s)^2}\,\|u^\varepsilon\|^2_{L^\infty(B)}\,\int_{B_t} \mathcal{G}_\varepsilon(\nabla u^\varepsilon)^\frac{2}{p_N'}\,dx\\
&+ C\,\|u^\varepsilon\|_{L^\infty(B)}^2\,\int_{B_R} |f^\varepsilon|^2\,dx\\
&\le C\,|B_R|+\frac{C}{(t-s)^2}\,\|u^\varepsilon\|^2_{L^\infty(B)}\,\int_{B_t} \mathcal{G}_\varepsilon(\nabla u^\varepsilon)^\frac{2}{p_N'}\,dx\\
&+ C\,|B_R|^{1-\frac{2}{\gamma}}\,\|u^\varepsilon\|_{L^\infty(B)}^2\,\left(\int_{B_R} |f^\varepsilon|^\gamma\,dx\right)^\frac{2}{\gamma}.
\end{split}
\]
for $C=C(N,p_1,p_N)>0$. We spend this information into \eqref{chiara} and use the subadditivity of $\tau\mapsto \tau^{(p_N-2)/p_N}$,
so to get
\[
\begin{split}
\int_{B_s} \mathcal{G}_\varepsilon(\nabla u^\varepsilon)^\frac{2}{p_N'}\,dx&\le \left(\int_{B_s} \mathcal{G}_\varepsilon(\nabla u^\varepsilon)\,dx\right)^\frac{2}{p_N}\,\left(C\,|B_R|\right)^\frac{p_N-2}{p_N}\\
&+\left(\int_{B_s} \mathcal{G}_\varepsilon(\nabla u^\varepsilon)\,dx\right)^\frac{2}{p_N}\,\left(\frac{C}{(t-s)^2}\,\|u^\varepsilon\|^2_{L^\infty(B)}\right)^\frac{p_N-2}{p_N}\,\left(\int_{B_t} \mathcal{G}_\varepsilon(\nabla u^\varepsilon)^\frac{2}{p_N'}\,dx\right)^\frac{p_N-2}{p_N}\\
&+\left(\int_{B_s} \mathcal{G}_\varepsilon(\nabla u^\varepsilon)\,dx\right)^\frac{2}{p_N}\,\left(C\,|B_R|^{1-\frac{2}{\gamma}}\,\left(\|u^\varepsilon\|_{L^\infty(B)}\,\|f^\varepsilon\|_{L^\gamma(B_R)}\right)^2\right)^\frac{p_N-2}{p_N}.
\end{split}
\]
For the second term on the right-hand side, we apply Young's inequality with conjugate exponents $p_N/2$ and $p_N/(p_N-2)$.
We get
\[
\begin{split}
\int_{B_s} \mathcal{G}_\varepsilon(\nabla u^\varepsilon)^\frac{2}{p_N'}\,dx&\le \left(\int_{B_R} \mathcal{G}_\varepsilon(\nabla u^\varepsilon)\,dx\right)^\frac{2}{p_N}\,\left(C\,|B_R|\right)^\frac{p_N-2}{p_N}\\
&+\frac{2}{p_N}\,\int_{B_R} \mathcal{G}_\varepsilon(\nabla u^\varepsilon)\,dx\,\left(\frac{C}{(t-s)^2}\,\|u^\varepsilon\|^2_{L^\infty(B)}\right)^\frac{p_N-2}{2}\\
&+\frac{p_N-2}{p_N}\,\int_{B_t} \mathcal{G}_\varepsilon(\nabla u^\varepsilon)^\frac{2}{p_N'}\,dx\\
&+\left(\int_{B_R} \mathcal{G}_\varepsilon(\nabla u^\varepsilon)\,dx\right)^\frac{2}{p_N}\,\left(C\,|B_R|^{1-\frac{2}{\gamma}}\,\left(\|u^\varepsilon\|_{L^\infty(B)}\,\|f^\varepsilon\|_{L^\gamma(B_R)}\right)^2\right)^\frac{p_N-2}{p_N}.
\end{split}
\]
We now use Lemma \ref{lm:giusti} with the choices
\[
Z(s)=\int_{B_s} \mathcal{G}_\varepsilon(\nabla u^\varepsilon)^\frac{2}{p_N'}\,dx,\qquad \vartheta=\frac{p_N-2}{p_N},\qquad \mathcal{A}=\frac{2}{p_N}\,\int_{B_R} \mathcal{G}_\varepsilon(\nabla u^\varepsilon)\,dx\,\left(C\,\|u^\varepsilon\|^2_{L^\infty(B)}\right)^\frac{p_N-2}{2},
\]
\[
\alpha_0=p_N-2,\qquad \mathcal{B}=0,
\]
and
\[
\mathcal{C}= \left(\int_{B_R} \mathcal{G}_\varepsilon(\nabla u^\varepsilon)\,dx\right)^\frac{2}{p_N}\,\left[\left(C\,|B_R|\right)^\frac{p_N-2}{p_N}+\left(C\,|B_R|^{1-\frac{2}{\gamma}}\,\left(\|u^\varepsilon\|_{L^\infty(B)}\,\|f^\varepsilon\|_{L^\gamma(B_R)}\right)^2\right)^\frac{p_N-2}{p_N}\right],
\]
in order to absorb the penultimate integral.
This permits to conclude that 
\[
\begin{split}
\int_{B_\varrho} \mathcal{G}_\varepsilon(\nabla u^\varepsilon)^\frac{2}{p_N'}\,dx&\le \widetilde{C}\,\left(\int_{B_R} \mathcal{G}_\varepsilon(\nabla u^\varepsilon)\,dx\right)^\frac{2}{p_N}\,\left(C\,|B_R|\right)^\frac{p_N-2}{p_N}\\
&+\widetilde{C}\,\frac{2}{p_N}\, \frac{1}{(R-r)^{p_N-2}}\int_{B_R} \mathcal{G}_\varepsilon(\nabla u^\varepsilon)\,dx\,\left(C\,\|u^\varepsilon\|^2_{L^\infty(B)}\right)^\frac{p_N-2}{2}\\
&+\widetilde{C}\,\left(\int_{B_R} \mathcal{G}_\varepsilon(\nabla u^\varepsilon)\,dx\right)^\frac{2}{p_N}\,\left(|B_R|^{1-\frac{2}{\gamma}}\,\left(\|u^\varepsilon\|_{L^\infty(B)}\,\|f^\varepsilon\|_{L^\gamma(B_R)}\right)^2\right)^\frac{p_N-2}{p_N}.
\end{split}
\]
This gives the claimed $L^{2/p_N'}-L^1$ estimate on $\mathcal{G}_\varepsilon(\nabla u^\varepsilon)$. The desired conclusion now easily follows. We leave the details to the reader.
\end{proof}
\begin{oss}[Quality of the constants]
\label{oss-gamma12}
For future references, it is important to notice that the two constants $\Gamma_1$ and $\Gamma_2$ in the previous statement are uniformly bounded from above, whenever there exists a constant $C\ge 1$ such that
\[
\|f^\varepsilon\|_{L^\gamma(B_R)}+\|u^\varepsilon\|_{L^\infty(B)}\le C,\qquad \mbox{ for every } 0<\varepsilon\le \varepsilon_0,
\]
and
\[
R-r\ge \frac{1}{C}.
\] 
On the contrary, we see from the proof above that
\[
\lim_{\gamma\to \infty} \Gamma_i=+\infty,\qquad \mbox{ for } i=1,2,
\]
with an exponential rate of divergence.
\end{oss}

\section{Uniform Lipschitz bound}\label{sec:Lip}
We now establish a local $L^\infty$ estimate for $\nabla u^\varepsilon$: this time, this will lead to Theorem L.
\begin{prop}
\label{prop:lipschitz}
Let $1<p_1\le \dots\le p_N\le 2$ and $0<\varepsilon\le \varepsilon_0$. Then for every pair of concentric balls $B_{r}\Subset B_{R}\Subset B$ and every $q>N$, 
we have
\[
\begin{split}
\|\mathcal{G}_\varepsilon(\nabla u^\varepsilon)\|_{L^\infty(B_{r})}&\le C\,\left[\frac{1}{(R-r)^\frac{N\,q}{q-N}}\,\left(\int_{B_{R}} \mathcal{G}_\varepsilon(\nabla u^\varepsilon)^q\,dx\right)^\frac{N}{q-N}+\|f^\varepsilon\|_{L^q(B_{R})}^\frac{N\,q}{q-N}\right]\,\left\|\mathcal{G}_\varepsilon(\nabla u^\varepsilon)\right\|_{L^1(B_{R})},
\end{split}
\]
for some $C=C(N,p_N,p_1,q)>0$.
\end{prop}
\begin{proof}
We will use a Moser's iteration scheme, in order to get the claimed estimate.
By \eqref{g2}, for every $i=1,\dots,N$ we have 
\[
\begin{split}
g''_{i,\varepsilon}(u_{x_i}^\varepsilon)\ge 
(p_i-1)\,(\varepsilon+|u^\varepsilon_{x_i}|^2)^\frac{p_i-2}{2}&=(p_i-1)\,\Big(p_i\,g_{i,\varepsilon}(u^\varepsilon_{x_i})\Big)^\frac{p_i-2}{p_i}\\
&\ge (p_i-1)\,\Big(p_i\,G_{\varepsilon}(\nabla u^\varepsilon)\Big)^\frac{p_i-2}{p_i}\ge (p_i-1)\,\Big(p_i\,\mathcal{G}_\varepsilon(\nabla u^\varepsilon)\Big)^\frac{p_i-2}{p_i},
\end{split}
\]
thanks to the fact that $p_i\le 2$, for every $i=1,\dots,N$.
We can further estimate the last term from below as follows
\[
g''_{i,\varepsilon}(u_{x_i}^\varepsilon)\ge (p_1-1)\,(p_N)^\frac{p_1-2}{p_1}\,\mathcal{G}_\varepsilon(\nabla u^\varepsilon)^\frac{p_1-2}{p_1}.
\]
By using this lower bound in \eqref{caccioppola}, we get
\begin{equation}
\label{moser1}
\begin{split}
\sum_{i=1}^N\int \mathcal{G}_\varepsilon(\nabla u^\varepsilon)^\frac{p_1-2}{p_1}&\,\left|\Big(\mathcal{G}_\varepsilon(\nabla u^\varepsilon)^\frac{\alpha+1}{2}\Big)_{x_i}\right|^2\,\eta^2\,dx\\&\le C\,(\alpha+1)^2\int \mathcal{G}_\varepsilon(\nabla u^\varepsilon)^{\alpha+2-\frac{2}{p_N}}\,\Big(|\nabla\eta|^2+\eta\,|D^2 \eta|\Big)\,dx\\
&+C\,(\alpha+1)^2\,\int |f^\varepsilon|^2\, \mathcal{G}_\varepsilon(\nabla u^\varepsilon)^\alpha\,\eta^2\,dx,
\end{split}
\end{equation}
for some $C=C(N,p_1,p_N)>0$. With simple algebraic manipulations, for every $\alpha\ge 0$
we have
\[
\begin{split}
\mathcal{G}_\varepsilon(\nabla u^\varepsilon)^\frac{p_1-2}{p_1}\,\left|\Big(\mathcal{G}_\varepsilon(\nabla u^\varepsilon)^\frac{\alpha+1}{2}\Big)_{x_i}\right|^2&=\left(\frac{\alpha+1}{2}\right)^2\,\frac{1}{\left(\dfrac{\alpha+2}{2}-\dfrac{1}{p_1}\right)^2}\,\left|\Big(\mathcal{G}_\varepsilon(\nabla u^\varepsilon)^{\frac{\alpha+2}{2}-\frac{1}{p_1}}\Big)_{x_i}\right|^2\\
&\ge \,\left|\Big(\mathcal{G}_\varepsilon(\nabla u^\varepsilon)^{\frac{\alpha+2}{2}-\frac{1}{p_1}}\Big)_{x_i}\right|^2.
\end{split}
\]
Thus from \eqref{moser1} we obtain
\begin{equation}
\label{moser2}
\begin{split}
\int \left|\nabla\Big(\mathcal{G}_\varepsilon(\nabla u^\varepsilon)^{\frac{\alpha+2}{2}-\frac{1}{p_1}}\Big)\right|^2\,\eta^2\,dx&\le C\,(\alpha+1)^2\int \mathcal{G}_\varepsilon(\nabla u^\varepsilon)^{\alpha+2-\frac{2}{p_N}}\,\Big(|\nabla\eta|^2+\eta\,|D^2 \eta|\Big)\,dx\\
&+C\,(\alpha+1)^2\,\int |f^\varepsilon|^2\, \mathcal{G}_\varepsilon(\nabla u^\varepsilon)^\alpha\,\eta^2\,dx,
\end{split}
\end{equation}
for some $C=C(N,p_1,p_N)>0$. By adding on both sides the term
\[
\int \Big(\mathcal{G}_\varepsilon(\nabla u^\varepsilon)^{\frac{\alpha+2}{2}-\frac{1}{p_1}}\Big)^2\,\left|\nabla\eta\right|^2\,dx,
\]
and using again that $\mathcal{G}_{\varepsilon}\ge 1$, we then obtain from \eqref{moser2}
\begin{equation}
\label{moser3}
\begin{split}
\int \left|\nabla\left(\Big(\mathcal{G}_\varepsilon(\nabla u^\varepsilon)^{\frac{\alpha+2}{2}-\frac{1}{p_1}}\Big)\,\eta\right)\right|^2\,dx&\le C\,(\alpha+1)^2\int \mathcal{G}_\varepsilon(\nabla u^\varepsilon)^{\alpha+2-\frac{2}{p_N}}\,\Big(|\nabla\eta|^2+\eta\,|D^2 \eta|\Big)\,dx\\
&+C\,(\alpha+1)^2\,\int |f^\varepsilon|^2\, \mathcal{G}_\varepsilon(\nabla u^\varepsilon)^\alpha\,\eta^2\,dx,
\end{split}
\end{equation}
possibly for a different constant $C=C(N,p_1,p_N)>0$. Let us suppose for simplicity that $N\ge 3$. The case $N=2$ can be treated with minor modifications. On the left-hand side of \eqref{moser3}, we then use Sobolev's inequality in $W^{1,2}(\mathbb{R}^N)$. This gives
\begin{equation}
\label{moser4}
\begin{split}
\left(\int \left(\mathcal{G}_\varepsilon(\nabla u^\varepsilon)^{\frac{\alpha+2}{2}-\frac{1}{p_1}}\,\eta\right)^{2^*}\,dx\right)^\frac{2}{2^*}&\le C\,(\alpha+1)^2\int \mathcal{G}_\varepsilon(\nabla u^\varepsilon)^{\alpha+2-\frac{2}{p_N}}\,\Big(|\nabla\eta|^2+\eta\,|D^2 \eta|\Big)\,dx\\
&+C\,(\alpha+1)^2\,\int |f^\varepsilon|^2\, \mathcal{G}_\varepsilon(\nabla u^\varepsilon)^\alpha\,\eta^2\,dx,
\end{split}
\end{equation}
with some new constant \(C=C(N,p_1,p_N)>0\).
We choose $\eta\in C^2_0(B_R)$ to be a cut-off function such that
\[
0\le \eta\le 1,\qquad \eta \equiv 1 \mbox{ on } B_r, \qquad |\nabla \eta|^2+|D^2 \eta|\le \frac{C}{(R-r)^2}.
\]
Thus we obtain from \eqref{moser4}
\[
\begin{split}
\left(\int_{B_r} \Big(\mathcal{G}_\varepsilon(\nabla u^\varepsilon)^{\frac{\alpha+2}{2}-\frac{1}{p_1}}\Big)^{2^*}\,dx\right)^\frac{2}{2^*}&\le C\,\frac{(\alpha+1)^2}{(R-r)^2}\int_{B_R} \mathcal{G}_\varepsilon(\nabla u^\varepsilon)^{\alpha+2-\frac{2}{p_N}}\,dx\\
&+C\,(\alpha+1)^2\,\int_{B_R} |f^\varepsilon|^2\, \mathcal{G}_\varepsilon(\nabla u^\varepsilon)^\alpha\,dx.
\end{split}
\]
We now take an exponent $q>N$. By H\"older's inequality on the last term, we deduce that
\begin{equation}
\label{moser5b}
\begin{split}
\left(\int_{B_r} \Big(\mathcal{G}_\varepsilon(\nabla u^\varepsilon)^{\frac{\alpha+2}{2}-\frac{1}{p_1}}\Big)^{2^*}\,dx\right)^\frac{2}{2^*}&\le C\,\frac{(\alpha+1)^2}{(R-r)^2}\,\int_{B_R} \mathcal{G}_\varepsilon(\nabla u^\varepsilon)^{\alpha+2-\frac{2}{p_N}}\,dx\\
&+C\,(\alpha+1)^2\,\|f^\varepsilon\|_{L^q(B_R)}^2\, \left(\int_{B_R}\mathcal{G}_\varepsilon(\nabla u^\varepsilon)^{\alpha\,\frac{q}{q-2}}\,dx\right)^\frac{q-2}{q}.
\end{split}
\end{equation}
Before proceeding further, we rely on H\"older's inequality to get
\[
\int_{B_R} \mathcal{G}_\varepsilon(\nabla u^\varepsilon)^{\alpha+2-\frac{2}{p_N}}\,dx\le \left(\int_{B_R} \mathcal{G}_\varepsilon(\nabla u^\varepsilon)^{\alpha\,\frac{q}{q-2}}\,dx\right)^\frac{q-2}{q}\,\left(\int_{B_R} \mathcal{G}_\varepsilon(\nabla u^\varepsilon)^{\frac{q}{p_N'}}\,dx\right)^\frac{2}{q}.
\]
Moreover, by recalling that $\mathcal{G}_{\varepsilon}\ge 1$, we have 
\[
\mathcal{G}_\varepsilon(\nabla u^\varepsilon)^{\alpha\,\frac{q}{q-2}}\le \Big(\mathcal{G}_\varepsilon(\nabla u^\varepsilon)^{\alpha+2-\frac{2}{p_1}}\Big)^\frac{q}{q-2}.
\]
By using these two facts in \eqref{moser5b}, we obtain
\begin{equation}
\label{moser6}
\begin{split}
\left(\int_{B_r} \mathcal{G}_\varepsilon(\nabla u^\varepsilon)^{\left(\alpha+2-\frac{2}{p_1}\right)\,\frac{2^*}{2}}\,dx\right)^\frac{2}{2^*}
&\le C\,(\alpha+1)^2\\
&\times\,\left[\frac{1}{(R-r)^2}\,\left(\int_{B_R} \mathcal{G}_\varepsilon(\nabla u^\varepsilon)^{\frac{q}{p_N'}}\,dx\right)^\frac{2}{q}+\|f^\varepsilon\|_{L^q(B_R)}^2\right]\\
&\times\left(\int_{B_R}\mathcal{G}_\varepsilon(\nabla u^\varepsilon)^{\left(\alpha+2-\frac{2}{p_1}\right)\,\frac{q}{q-2}}\,dx\right)^\frac{q-2}{q}.
\end{split}
\end{equation}
We now set 
\[
\theta=\alpha+2-\frac{2}{p_1},
\]
so that from \eqref{moser6}, we get
\begin{equation}
\label{moser7}
\begin{split}
\left(\int_{B_r}\mathcal{G}_\varepsilon(\nabla u^\varepsilon)^{\frac{2^*}{2}\theta}\,dx\right)^\frac{2}{2^*\,\theta}&\le \left(C\theta^2\right)^{\frac{1}{\theta}}\,\left[\frac{1}{(R-r)^2}\,\left(\int_{B_R} \mathcal{G}_\varepsilon(\nabla u^\varepsilon)^{\frac{q}{p_N'}}\,dx\right)^\frac{2}{q}+\|f^\varepsilon\|_{L^q(B_R)}^2\right]^\frac{1}{\theta}\\
&\times\left(\int_{B_R}\mathcal{G}_\varepsilon(\nabla u^\varepsilon)^{\theta\frac{q}{q-2}}\,dx\right)^\frac{q-2}{q\,\theta}.
\end{split}
\end{equation}
We define the sequence of exponents through the following recursive relation
\[
\theta_0=1,\qquad \theta_{i+1}\,\frac{q}{q-2}=\frac{2^*}{2}\,\theta_i,\qquad \mbox{ for }i\in\mathbb{N};
\]
that is\footnote{Observe that 
\[
\frac{2\,q}{2^*\,(q-2)}<1 \qquad \Longleftrightarrow \qquad q>N,
\]
and the latter holds true, in view of our assumption.},
\[
\theta_{i+1}=\left(\frac{2^*}{2}\,\frac{q-2}{q}\right)\,\theta_i=\left(\frac{2^*}{2}\,\frac{q-2}{q}\right)^{i+1},\qquad \mbox{ for } i\in\mathbb{N}.
\]
We also define the classical sequence of shrinking radii
\[
r_i=r+\frac{R-r}{2^i},\qquad \mbox{ for } i\in\mathbb{N}.
\]
With this notation, from \eqref{moser7} we get 
\begin{equation}
\label{moser8}
\begin{split}
\left\|\mathcal{G}_\varepsilon(\nabla u^\varepsilon)\right\|_{L^{\frac{q}{q-2}\,\theta_{i+1}}(B_{r_{i+1}})}&\le \left[\frac{C}{(R-r)^2}\,\left(\int_{B_R} \mathcal{G}_\varepsilon(\nabla u^\varepsilon)^{\frac{q}{p_N'}}\,dx\right)^\frac{2}{q}+C\,\|f^\varepsilon\|_{L^q(B_R)}^2\right]^\frac{1}{\theta_i}\\
&\times (4^i\,\theta_i^2)^\frac{1}{\theta_i}\,
\left\|\mathcal{G}_\varepsilon(\nabla u^\varepsilon)\right\|_{L^{\frac{q}{q-2}\,\theta_{i}}(B_{r_{i}})}.
\end{split}
\end{equation}
By starting from $i=0$ and iterating \eqref{moser8} $n$ times, we get
\begin{equation}
\label{moser9}
\begin{split}
\left\|\mathcal{G}_\varepsilon(\nabla u^\varepsilon)\right\|_{L^{\frac{q}{q-2}\,\theta_{n+1}}(B_{r_{n+1}})}&\le 
\left[\frac{C}{(R-r)^2}\,\left(\int_{B_R} \mathcal{G}_\varepsilon(\nabla u^\varepsilon)^{\frac{q}{p_N'}}\,dx\right)^\frac{2}{q}+C\,\|f^\varepsilon\|_{L^q(B_R)}^2\right]^{\sum\limits_{i=0}^n\frac{1}{\theta_i}}\\
&\times\prod_{i=0}^n (4^i\,\theta_i^2)^\frac{1}{\theta_i}\,
\left\|\mathcal{G}_\varepsilon(\nabla u^\varepsilon)\right\|_{L^{\frac{q}{q-2}}(B_{R})}.
\end{split}
\end{equation}
By observing that 
\[
\lim_{n\to\infty}\prod_{i=0}^n (4^i\,\theta_i^2)^\frac{1}{\theta_i}=:C_{N,q}<+\infty,
\]
and
\[
\lim_{n\to\infty}\sum\limits_{i=0}^n\frac{1}{\theta_i}=\sum_{i=0}^\infty \left(\frac{2\,q}{2^*\,(q-2)}\right)^i=\frac{N}{2}\,\frac{q-2}{q-N},
\]
if we take the limit as $n$ goes to $\infty$ in \eqref{moser9}, we end up with
\[
\begin{split}
\|\mathcal{G}_\varepsilon(\nabla u^\varepsilon)\|_{L^\infty(B_r)}&\le C\,\left[\frac{1}{(R-r)^2}\,\left(\int_{B_R} \mathcal{G}_\varepsilon(\nabla u^\varepsilon)^{\frac{q}{p_N'}}\,dx\right)^\frac{2}{q}+\|f^\varepsilon\|_{L^q(B_R)}^2\right]^{\frac{N}{2}\,\frac{q-2}{q-N}}\\
&\times\left\|\mathcal{G}_\varepsilon(\nabla u^\varepsilon)\right\|_{L^{\frac{q}{q-2}}(B_{R})},
\end{split}
\]
for some $C=C(N,p_1,p_N,q)>0$. The previous estimate holds for every $r<R$ such that $B_R\Subset B$. Thus, we can now use a standard interpolation trick to rectify it and replace the $L^{q/(q-2)}$ norm on the right-hand side by the $L^1$ norm.
\par
This goes as follows: we first observe that
\[
\left\|\mathcal{G}_\varepsilon(\nabla u^\varepsilon)\right\|_{L^{\frac{q}{q-2}}(B_{R})}\le \left(\left\|\mathcal{G}_\varepsilon(\nabla u^\varepsilon)\right\|_{L^\infty(B_{R})}\right)^\frac{2}{q}\,\left(\left\|\mathcal{G}_\varepsilon(\nabla u^\varepsilon)\right\|_{L^1(B_{R})}\right)^\frac{q-2}{q}.
\]
Then by using Young's inequality with exponents $q/2$ and $q/(q-2)$, we get
\[
\begin{split}
\|\mathcal{G}_\varepsilon(\nabla u^\varepsilon)\|_{L^\infty(B_r)}&\le \frac{q-2}{q}\,C^\frac{q}{q-2}\,\left[\frac{1}{(R-r)^2}\,\left(\int_{B_R} \mathcal{G}_\varepsilon(\nabla u^\varepsilon)^{\frac{q}{p_N'}}\,dx\right)^\frac{2}{q}+\|f\|_{L^q(B_R)}^2\right]^{\frac{N}{2}\frac{q}{q-N}}\\
&\times\left\|\mathcal{G}_\varepsilon(\nabla u^\varepsilon)\right\|_{L^1(B_{R})}+\frac{2}{q}\,\left\|\mathcal{G}_\varepsilon(\nabla u^\varepsilon)\right\|_{L^\infty(B_{R})}.
\end{split}
\]
We now take $s,t$ such that $r\le s<t\le R$. The previous estimate is valid by replacing $r$ with $s$ and $R$ with $t$. Thus, with some simple algebraic manipulations, we get
\[
\begin{split}
\|\mathcal{G}_\varepsilon(\nabla u^\varepsilon)\|_{L^\infty(B_s)}\le \frac{q-2}{q}\,\frac{C^\frac{q}{q-2}}{2^{\frac{q(N-2)+2N}{2(N-q)}}}&\left[\frac{1}{(t-s)^\frac{N\,q}{q-N}}\,\left(\int_{B_R} \mathcal{G}_\varepsilon(\nabla u^\varepsilon)^{\frac{q}{p_N'}}\,dx\right)^\frac{N}{q-N}+\|f^\varepsilon\|_{L^q(B_R)}^\frac{N\,q}{q-N}\right]\\
&\times\left\|\mathcal{G}_\varepsilon(\nabla u^\varepsilon)\right\|_{L^1(B_{R})}+\frac{2}{q}\,\left\|\mathcal{G}_\varepsilon(\nabla u^\varepsilon)\right\|_{L^\infty(B_{R})}.
\end{split}
\]
By relying once again on Lemma \ref{lm:giusti}, from the last estimate we get
\[
\begin{split}
\|\mathcal{G}_\varepsilon(\nabla u^\varepsilon)\|_{L^\infty(B_r)}&\le \widetilde{C}\,\left[\frac{1}{(R-r)^\frac{N\,q}{q-N}}\,\left(\int_{B_R} \mathcal{G}_\varepsilon(\nabla u^\varepsilon)^{\frac{q}{p_N'}}\,dx\right)^\frac{N}{q-N}+\|f^\varepsilon\|_{L^q(B_R)}^\frac{N\,q}{q-N}\right]\,\left\|\mathcal{G}_\varepsilon(\nabla u^\varepsilon)\right\|_{L^1(B_{R})}.
\end{split}
\]
By finally using that 
\[
\mathcal{G}_\varepsilon(\nabla u^\varepsilon)^\frac{q}{p_N'}\le \mathcal{G}_\varepsilon(\nabla u^\varepsilon)^q,
\]
we eventually conclude the proof.
\end{proof}

\section{Uniform higher differentiability}\label{diff}

At last, we prove a Sobolev--type regularity result for (some nonlinear function of) $\nabla u^\varepsilon$, which eventually will permit to establish Theorem S.
\begin{prop}
Let $1<p_1\le \dots\le p_N\le 2$. For $0<\varepsilon\le \varepsilon_0$ and $i=1,\dots,N$, we set
\begin{equation}
\label{Vi}
\mathcal{V}_{i,\varepsilon}=V_{i,\varepsilon}(u^\varepsilon_{x_i}),\qquad \mbox{ with } V_{i,\varepsilon}(t)=\int_0^t \sqrt{g_{i,\varepsilon}''(\tau)}
\,d\tau.
\end{equation}
Then for every non-negative $\eta\in C^2_0(B)$ and every \(\gamma\geq 2\), we have
\begin{equation}
\label{aprioriV}
\begin{split}
\sum_{i=1}^N \int \Big|\nabla \mathcal{V}_{i,\varepsilon}\Big|^2 \eta^2\,dx&\leq  C\,\int \mathcal{G}_\varepsilon(\nabla u^\varepsilon)^{1+2\,\left(\frac{1}{p_1}-\frac{1}{p_N}\right)}\, \Big(|\nabla \eta|^2+ \eta\,|D^2\eta|\Big)\,dx\\
& +C\,\left(\int \mathcal{G}_\varepsilon(\nabla u^\varepsilon)^{\frac{2-p_1}{p_1}\frac{\gamma}{\gamma-2}}\,\eta^2\,dx\right)^\frac{\gamma-2}{\gamma}\,\left(\int |f^\varepsilon|^\gamma\,\eta^2\,dx\right)^\frac{2}{\gamma},
\end{split}
\end{equation}
for some $C=C(N,p_N,p_1)>0$.
\end{prop}
When \(\gamma=2\), the last term is simply \(C\,\|f^\varepsilon \eta\|_{L^{2}}^2\).
\begin{proof}
We start by fixing $k\in\{1,\dots,N\}$ and inserting in the differentiated equation \eqref{Eulerdiff} the test function
$\varphi=u^\varepsilon_{x_k}\,\eta^2$.
Thus we get 
\[
\begin{split}
\sum_{i=1}^N \int g_{i,\varepsilon}''(u^\varepsilon_{x_i})\,|u^\varepsilon_{x_k\,x_i}|^2\, \eta^2\,dx
&= -2\,\sum_{i=1}^{N}\int g_{i,\varepsilon}''(u^\varepsilon_{x_i})\,u^\varepsilon_{x_k\,x_i}\,u^\varepsilon_{x_k}\,\eta\, \eta_{x_i}\,dx\\ 
&+\int (f^\varepsilon (\zeta^\varepsilon)'(u^\varepsilon))_{x_k}\,u^{\varepsilon}_{x_k}\,\eta^2\,dx. 
\end{split}
\]
For the first term of the right-hand side, we use the same trick as in the proof of Proposition \ref{prop:caccioespilon2}:
we observe that 
\[
g_{i,\varepsilon}''(u^\varepsilon_{x_i})\,u^\varepsilon_{x_k\,x_i}=\big(g_{i,\varepsilon}'(u^\varepsilon_{x_i})\big)_{x_k},
\]
and then integrate by parts. We integrate by parts the term $(f^\varepsilon (\zeta^\varepsilon)'(u^\varepsilon))_{x_k}$, as well. This yields
\begin{equation}
\label{2hd}
\begin{split}
\sum_{i=1}^N \int g_{i,\varepsilon}''(u^\varepsilon_{x_i})\,|u^\varepsilon_{x_k\,x_i}|^2\, \eta^2\,dx
&= 2\,\sum_{i=1}^{N}\int g_{i,\varepsilon}'(u^\varepsilon_{x_i})\,(u^\varepsilon_{x_k}\,\eta\, \eta_{x_i})_{x_k}-\int f^\varepsilon (\zeta^\varepsilon)'(u^\varepsilon)\,(u^{\varepsilon}_{x_k}\,\eta^2)_{x_k}\,dx\\
&=\int u_{x_k\,x_k}^\varepsilon \left(2\,\sum_{i=1}^{N} g_{i,\varepsilon}'(u^\varepsilon_{x_i})\,\eta\, \eta_{x_i} -f^\varepsilon (\zeta^\varepsilon)'(u^\varepsilon)\,\eta^2\right)\,dx\\
&+2\,\int u^\varepsilon_{x_k}\,\left(\sum_{i=1}^{N} g_{i,\varepsilon}'(u^\varepsilon_{x_i})\,(\eta_{x_k}\, \eta_{x_i}+\eta\,\eta_{x_i\,x_k})-\eta\,\eta_{x_k}\,f^\varepsilon (\zeta^\varepsilon)'(u^\varepsilon) \right)\,dx.
\end{split}
\end{equation}
By Young's inequality and the fact that $0\le (\zeta^\varepsilon)'\le 1$, we can estimate the first term of the right-hand side as follows:
\[
\begin{split}
\int u_{x_k\,x_k}^\varepsilon &\left(2\,\sum_{i=1}^{N} g_{i,\varepsilon}'(u^\varepsilon_{x_i})\,\eta\, \eta_{x_i} -f^\varepsilon (\zeta^\varepsilon)'(u^\varepsilon)\,\eta^2\right)\,dx
\leq \frac{1}{2}\,\int g_{k,\varepsilon}''(u^\varepsilon_{x_k})\,|u^\varepsilon_{x_k\,x_k}|^2\, \eta^2\,dx\\
&+ 4\,\int \frac{1}{g_{k,\varepsilon}''(u^\varepsilon_{x_k})}\,\left(\left(\sum_{i=1}^{N} |g_{i,\varepsilon}'(u^\varepsilon_{x_i})|\right)^2 |\nabla \eta|^2 +|f^\varepsilon|^2\,\eta^2\right)\,dx.
\end{split}
\]
We use \eqref{g2} to estimate $1/g_{k,\varepsilon}''$ on the right-hand side.
On account of this inequality, we get
\[
\begin{split}
\int u_{x_k\,x_k}^\varepsilon\,& \left(2\,\sum_{i=1}^{N} g_{i,\varepsilon}'(u^\varepsilon_{x_i})\,\eta\, \eta_{x_i} -f^\varepsilon (\zeta^\varepsilon)'(u^\varepsilon)\,\eta^2\right)\,dx
\leq \frac{1}{2}\,\int g_{k,\varepsilon}''(u^\varepsilon_{x_k})\,|u^\varepsilon_{x_k\,x_k}|^2\, \eta^2\,dx\\
&+ \frac{4}{p_k-1}\,\int \left(\varepsilon+(u^\varepsilon_{x_k})^2\right)^{\frac{2-p_k}{2}}\,\left(\left(\sum_{i=1}^{N} |g_{i,\varepsilon}'(u^\varepsilon_{x_i})|\right)^2 |\nabla \eta|^2 +|f^\varepsilon|^2\,\eta^2\right)\,dx.
\end{split}
\]
Inserting the above inequality into \eqref{2hd} and absorbing the Hessian term of the right-hand side into the left-hand side, one gets
\begin{equation}
\label{3hd}
\begin{split}
\sum_{i=1}^N \int g_{i,\varepsilon}''(u^\varepsilon_{x_i})&\,|u^\varepsilon_{x_k\,x_i}|^2\, \eta^2\,dx\\
&\leq  \frac{8}{p_k-1}\,\int \left(\varepsilon+(u^\varepsilon_{x_k})^2\right)^{\frac{2-p_k}{2}}\,\left(\left(\sum_{i=1}^{N} |g_{i,\varepsilon}'(u^\varepsilon_{x_i})|\right)^2 |\nabla \eta|^2 +|f^\varepsilon|^2\,\eta^2\right)\,dx\\
&+4\,\int u^\varepsilon_{x_k}\,\left(\sum_{i=1}^{N} g_{i,\varepsilon}'(u^\varepsilon_{x_i})\,(\eta_{x_k}\, \eta_{x_i}+\eta\,\eta_{x_i\,x_k})-\eta\,\eta_{x_k}\,f^\varepsilon (\zeta^\varepsilon)'(u^\varepsilon) \right)\,dx.
\end{split}
\end{equation}
We now estimate the last term as follows:
\[
\begin{split}
\left|\int u^\varepsilon_{x_k}\right.&\left.\left(\sum_{i=1}^{N} g_{i,\varepsilon}'(u^\varepsilon_{x_i})\,(\eta_{x_k}\, \eta_{x_i}+\eta\,\eta_{x_i\,x_k})-\eta\,\eta_{x_k}\,f^\varepsilon (\zeta^\varepsilon)'(u^\varepsilon) \right)\,dx\right|\\
&\leq \int |u^\varepsilon_{x_k}|\,|f^\varepsilon|\,\eta\,|\nabla \eta|\,dx +\int |u^\varepsilon_{x_k}|\,\sum_{i=1}^{N} |g_{i,\varepsilon}'(u^\varepsilon_{x_i})|\,\Big(|\nabla \eta|^2 + \eta\,|D^2\eta|\Big)\,dx.
\end{split}
\]
Observe that we used again that $0\le (\zeta^\varepsilon)'\le 1$.
We apply Young's inequality on the last term, so to obtain
\[
\begin{split}
\int |u^\varepsilon_{x_k}|\,\sum_{i=1}^{N} |g_{i,\varepsilon}'(u^\varepsilon_{x_i})|\Big(|\nabla \eta|^2 + \eta\,|D^2\eta|\Big)\,dx
&\leq \frac{1}{2}\int |u^\varepsilon_{x_k}|^{2-p_k}\!\left(\sum_{i=1}^{N} |g_{i,\varepsilon}'(u^\varepsilon_{x_i})\big|\right)^2\!\!\Big(|\nabla \eta|^2 + \eta\,|D^2\eta|) \Big)\,dx\\
& + \frac{1}{2}\int |u^\varepsilon_{x_k}|^{p_k}\,\Big(|\nabla \eta|^2 + \eta\,|D^2\eta|\Big)\,dx.
\end{split}
\]
By further using that $p_k\le 2$, we have
\[
|u^\varepsilon_{x_k}|^{2-p_k} \leq (\varepsilon+(u^\varepsilon_{x_k})^2)^{\frac{2-p_k}{2}}= p_k^\frac{2-p_k}{p_k}\,g_{k,\varepsilon}(u_{x_k}^\varepsilon)^{\frac{2-p_k}{p_k}}.
\]
It follows from the above inequality and \eqref{3hd} that
\[
\begin{split}
\sum_{i=1}^N \int g_{i,\varepsilon}''(u^\varepsilon_{x_i})\,|u^\varepsilon_{x_k\,x_i}|^2\, \eta^2\,dx&\leq  C\,\int g_{k,\varepsilon}(u_{x_k}^\varepsilon)^{\frac{2-p_k}{p_k}}\,\left(\left(\sum_{i=1}^{N} |g_{i,\varepsilon}'(u^\varepsilon_{x_i})|\right)^2 |\nabla \eta|^2 +|f^\varepsilon|^2\,\eta^2\right)\,dx\\\
&+C\,\int |u^\varepsilon_{x_k}|^{p_k}\,\Big(|\nabla \eta|^2 + \eta\,|D^2\eta|\Big)\,dx+C\,\int |u^\varepsilon_{x_k}|\,|f^\varepsilon|\,\eta\,|\nabla \eta|\,dx,
\end{split}
\]
for some $C=C(p_1,p_N)>0$.  Then take the sum over $k=1,\dots,N$. This gives
\begin{equation}
\label{5hd}
\begin{split}
\sum_{i=1}^N &\int g_{i,\varepsilon}''(u^\varepsilon_{x_i})\,|\nabla u^\varepsilon_{x_i}|^2\, \eta^2\,dx\\
&\leq  C\,\int \sum_{k=1}^Ng_{k,\varepsilon}(u_{x_k}^\varepsilon)^{\frac{2-p_k}{p_k}}\,\left(\left(\sum_{i=1}^{N} |g_{i,\varepsilon}'(u^\varepsilon_{x_i})|\right)^2 |\nabla \eta|^2 +|f^\varepsilon|^2\,\eta^2\right)\,dx\\\
&+C\,\int \sum_{k=1}^N|u^\varepsilon_{x_k}|^{p_k}\,\Big(|\nabla \eta|^2 + \eta\,|D^2\eta|\Big)\,dx+C\,\int \sum_{k=1}^N |u^\varepsilon_{x_k}|\,|f^\varepsilon|\,\eta\,|\nabla \eta|\,dx.
\end{split}
\end{equation}
By Lemma \ref{lm:majorette}, we have
\[
\sum_{i=1}^N|g_{i, \varepsilon}'(u_{x_i}^{\varepsilon})|\le p_N^\frac{p_N-1}{p_N}\,\sum_{i=1}^NG_\varepsilon(\nabla u^\varepsilon)^{\frac{p_i-1}{p_i}}\le N\,p_N^\frac{p_N-1}{p_N}\,\mathcal{G}_\varepsilon(\nabla u^\varepsilon)^{\frac{p_N-1}{p_N}}.
\] 
Moreover, by the definitions of $g_{k,\varepsilon},\, G_\varepsilon$, and $\mathcal{G}_\varepsilon$, it is easily seen that
\[
\sum_{k=1}^Ng_{k,\varepsilon}(u_{x_k}^\varepsilon)^{\frac{2-p_k}{p_k}}\le C\,\mathcal{G}_\varepsilon(\nabla u^\varepsilon)^\frac{2-p_1}{p_1},
\]
\[
\sum_{k=1}^N|u^\varepsilon_{x_k}|^{p_k}\le \sum_{k=1}^Np_k\, g_{k, \varepsilon}(u_{x_k}^\varepsilon)\le C\,\mathcal{G}_\varepsilon(\nabla u^\varepsilon)\qquad \mbox{ and }\qquad \sum_{k=1}^N |u^\varepsilon_{x_k}|\le C\,\mathcal{G}_\varepsilon(\nabla u^\varepsilon)^\frac{1}{p_1},
\]
where all the constants depend only on $N, p_1$ and $p_N$. 
From \eqref{5hd}, we get
\[
\begin{split}
\sum_{i=1}^N& \int g_{i,\varepsilon}''(u^\varepsilon_{x_i})\,|\nabla u^\varepsilon_{x_i}|^2\, \eta^2\,dx\\
&\leq  C\,\int \mathcal{G}_\varepsilon(\nabla u^\varepsilon)^{1+2\,\left(\frac{1}{p_1}-\frac{1}{p_N}\right)}\, |\nabla \eta|^2\,dx +C\,\int \mathcal{G}_\varepsilon(\nabla u^\varepsilon)^\frac{2-p_1}{p_1}\,|f^\varepsilon|^2\,\eta^2\,dx\\\
&+C\,\int \mathcal{G}_\varepsilon(\nabla u^\varepsilon)
\,\Big(|\nabla \eta|^2 + \eta\,|D^2\eta|\Big)\,dx+C\,\int \mathcal{G}_\varepsilon(\nabla u^\varepsilon)^\frac{1}{p_1}\,|f^\varepsilon|\,\eta\,|\nabla \eta|\,dx,
\end{split}
\]
for some $C=C(N,p_1,p_N)>0$. Since 
\[
1\leq 1+2\,\left(\frac{1}{p_1}-\frac{1}{p_N}\right) \qquad \mbox{ and }\qquad \mathcal{G}_\varepsilon(\nabla u^\varepsilon)\geq 1,
\] 
the third term can be absorbed in the first one, up to increasing \(C\) if necessary:
\begin{equation}
\label{7hd}
\begin{split}
\sum_{i=1}^N& \int g_{i,\varepsilon}''(u^\varepsilon_{x_i})\,|\nabla u^\varepsilon_{x_i}|^2\, \eta^2\,dx\\
&\leq  C\,\int \mathcal{G}_\varepsilon(\nabla u^\varepsilon)^{1+2\,\left(\frac{1}{p_1}-\frac{1}{p_N}\right)}\,\Big(|\nabla \eta|^2 + \eta\,|D^2\eta|\Big)\,dx +C\,\int \mathcal{G}_\varepsilon(\nabla u^\varepsilon)^\frac{2-p_1}{p_1}\,|f^\varepsilon|^2\,\eta^2\,dx\\\
&+C\,\int \mathcal{G}_\varepsilon(\nabla u^\varepsilon)^\frac{1}{p_1}\,|f^\varepsilon|\,\eta\,|\nabla \eta|\,dx.
\end{split}
\end{equation}
In the last term, we write 
\[
\mathcal{G}_\varepsilon(\nabla u^\varepsilon)^\frac{1}{p_1}\,|f^\varepsilon|\,\eta\,|\nabla \eta|
=\left( \mathcal{G}_\varepsilon(\nabla u^\varepsilon)^\frac{1}{2}|\nabla\eta|\right) \left(\mathcal{G}_\varepsilon(\nabla u^\varepsilon)^{\frac{1}{p_1}-\frac{1}{2}}\,|f^\varepsilon|\,\eta\right),
\]
and use Young's inequality:
\begin{align*}
\int \mathcal{G}_\varepsilon(\nabla u^\varepsilon)^\frac{1}{p_1}&\,|f^\varepsilon|\,\eta\,|\nabla \eta|\,dx \le \int \mathcal{G}_\varepsilon(\nabla u^\varepsilon)\,|\nabla \eta|^2\,dx + \int \mathcal{G}_\varepsilon(\nabla u^\varepsilon)^{\frac{2}{p_1}-1}\,|f^\varepsilon|^2\,\eta^2\,dx \\
&\le \int \mathcal{G}_\varepsilon(\nabla u^\varepsilon)^{1+2\,\left(\frac{1}{p_1}-\frac{1}{p_N}\right)}\,\Big(|\nabla \eta|^2+\eta\,|D^2\eta|\Big)\,dx + \int \mathcal{G}_\varepsilon(\nabla u^\varepsilon)^{\frac{2}{p_1}-1}\,|f^\varepsilon|^2\,\eta^2\,dx,
\end{align*}
where in the last line, we have used again that \(\mathcal{G}_\varepsilon\ge 1\). Inserting this estimate in \eqref{7hd}, we obtain
\[
\begin{split}
\sum_{i=1}^N &\int g_{i,\varepsilon}''(u^\varepsilon_{x_i})\,|\nabla u^\varepsilon_{x_i}|^2\, \eta^2\,dx\\
&\leq  C\,\int \mathcal{G}_\varepsilon(\nabla u^\varepsilon)^{1+2\,\left(\frac{1}{p_1}-\frac{1}{p_N}\right)}\,\Big(|\nabla \eta|^2 + \eta\,|D^2\eta|\Big)\,dx +C\,\int \mathcal{G}_\varepsilon(\nabla u^\varepsilon)^\frac{2-p_1}{p_1}\,|f^\varepsilon|^2\,\eta^2\,dx.
\end{split}
\]
On the left-hand side, we use the definition \eqref{Vi} of $\mathcal{V}_{i,\varepsilon}$ which gives that 
\[
|\nabla \mathcal{V}_{i,\varepsilon}|^2 =g''_{i,\varepsilon}(u^\varepsilon_{x_i})\,|\nabla u_{x_i}^\varepsilon|^2.
\]
This yields the the desired conclusion when the exponent \(\gamma\) in the statement of Proposition is equal to \(2\). When \(\gamma>2\), we only need to apply the H\"older inequality to the last term of the right-hand side with the exponents   
\(\gamma/(\gamma-2)\) and \(\gamma/2\). The proof is complete. 
\end{proof}

\section{Proofs of the main results}\label{proofmain}

We finally establish the three results presented in the Introduction by relying on the relevant a priori estimates that we have obtained in the previous sections.
We thus fix a ball $B_{4R}(x_0)\Subset\Omega$ as in the statements of Proposition \ref{prop:high}, Theorem L and Theorem S: we are going to use the results of the previous sections, with the choice $B=B_{2\,R}(x_0)$.
\par
We will use the functions $\mathcal{G}_0$ and $\mathcal{G}_\varepsilon$, defined by \eqref{G0} and \eqref{Ggrande}. Moreover, we will omit to indicate the centers of the balls, which will always be $x_0$. 
\subsection{Proof of Proposition \ref{prop:high}}

In this section, we assume that \(1<p_1\leq \dots \leq p_N<\infty\) and \(f\in L^{\gamma}_{\rm loc}(\Omega)\) for some \(\gamma\ge 2\). 
We consider the ball $B_R\Subset B$. Then for every \(0<\varepsilon<\varepsilon_0\),
\[
\|f^\varepsilon\|_{L^\gamma(B_R)}\leq \|f\|_{L^{\gamma}(2\,B)}.
\]
Using also that \(\|u^\varepsilon\|_{L^\infty(B)}\leq  M+1\),  Proposition \ref{prop:plc} and  Remark \ref{oss-gamma12} imply that  
for every \(0<\varepsilon <\varepsilon_0\) we have
\begin{equation}\label{eq1734}
\int_{B_\frac{R}{2}} \mathcal{G}_\varepsilon(\nabla u^\varepsilon)^\gamma\,dx\le \Gamma_1+\Gamma_2\,\int_{B_R} \mathcal{G}_\varepsilon(\nabla u^\varepsilon)\,dx,
\end{equation}
for two constants $\Gamma_1,\Gamma_2>0$ which do not depend on \(\varepsilon\), but only on 
\[
N,\, p_N,\, p_1,\,\gamma,\,R,\,\|f\|_{L^\gamma(2\,B)} \mbox{ and } M=\|U\|_{L^\infty(2\,B)}.
\]
In particular, by using that\footnote{The upper bound simply follows from \eqref{piccolina}, with standard algebraic manipulations.} 
\begin{equation}
\label{pointwiseG}
\mathcal{G}_0(z)\le \mathcal{G}_\varepsilon(z)\le C\,\left(\varepsilon^\frac{p_1}{2}+\mathcal{G}_0(z)\right),\qquad \mbox{ for every } z\in\mathbb{R}^N,
\end{equation}
for some $C=C(N,p_N,p_1)>0$, we can infer
\[
\int_{B_\frac{R}{2}} \mathcal{G}_0(\nabla u^\varepsilon)^\gamma\,dx\le \Gamma_1+C\,\Gamma_2\,\int_{B_R} \left(\varepsilon^\frac{p_1}{2}+\mathcal{G}_0(\nabla u^\varepsilon)\right)\,dx
\]
In view of Lemma \ref{lm:convergence},  there exists an infinitesimal sequence \(\{\varepsilon_k\}_{k\in\mathbb{N}}\) such that 
\[
(u^{\varepsilon_k}, \nabla u^{\varepsilon_k})\ \textrm{ converges to }\ (U,\nabla U)\qquad \textrm{ a.\,e. in } B.
\]
We then take the limit on both sides of the estimate above and use Fatou's lemma on the left. We get
\begin{equation}\label{eq1754}
\int_{B_\frac{R}{2}}\mathcal{G}_0(\nabla U)^\gamma\,dx\le \Gamma_1+C\,\Gamma_2\,\liminf_{k\to \infty}\int_{B_R} \mathcal{G}_0(\nabla u^{\varepsilon_k})\,dx.
\end{equation}
By Lemma \ref{lm:convergence}, the functions \(u_{x_i}^{\varepsilon}\) converge to \(U_{x_i}\) in \(L^{p_i}(B)\). Hence, the continuity of the map \(v\in L^{p_i}(B) \mapsto |v|^{p_i}\in L^{1}(B)\) implies that
\begin{equation}\label{eq-Vitali}
\lim_{\varepsilon\to 0}\left\|\mathcal{G}_0(\nabla u^{\varepsilon})-\mathcal{G}_0(\nabla U)\right\|_{L^{1}(B)}=0.
\end{equation}
By using this result in \eqref{eq1754}, we obtain
\[
\int_{B_\frac{R}{2}} \mathcal{G}_0(\nabla U)^{\gamma}\,dx\le \Gamma_1+C\,\Gamma_2\,\int_{B_R} \mathcal{G}_0(\nabla U)\,dx. 
\]
This concludes the proof, up to rename the constant $\Gamma_2$.

\subsection{Proof of Theorem L}

In this section, we assume that \(1<p_1\leq \dots \leq p_N\leq 2\) and \(f\in L^{\gamma}_{\rm loc}(\Omega)\) for some \(\gamma>N\).
In particular, \(\gamma\ge 2\) and thus we can rely on Proposition \ref{prop:plc}.
\par
We introduce the ball $B_{R}\Subset B$ as before.
By Proposition \ref{prop:lipschitz} applied with $B_{R/4}$ and $B_{R/2}$, for every \(\varepsilon\in (0, \varepsilon_0)\), we have
\[
\begin{split}
\|\mathcal{G}_\varepsilon(\nabla u^\varepsilon)\|_{L^\infty(B_\frac{R}{4})}&\le C\,\left[\left(\frac{4}{R}\right)^\frac{N\,\gamma}{\gamma-N}\,\left(\int_{B_\frac{R}{2}} \mathcal{G}_\varepsilon(\nabla u^\varepsilon)^\gamma\,dx\right)^\frac{N}{\gamma-N}+\|f^\varepsilon\|_{L^\gamma(B_{R})}^\frac{N\,\gamma}{\gamma-N}\right]\,\left\|\mathcal{G}_\varepsilon(\nabla u^\varepsilon)\right\|_{L^1(B_{R})},
\end{split}
\]
for some $C=C(N,p_N,p_1,\gamma)>0$. On the right-hand side, we can apply \eqref{eq1734}, in order to estimate the term containing $\mathcal{G}_\varepsilon^\gamma$. This yields
\[
\begin{split}
\|\mathcal{G}_\varepsilon(\nabla u^\varepsilon)\|_{L^\infty(B_\frac{R}{4})}&\le C\left[\left(\frac{4}{R}\right)^\frac{N\,\gamma}{\gamma-N}\!\!\left(\Gamma_1+\Gamma_2\int_{B_R} \mathcal{G}_\varepsilon(\nabla u^\varepsilon)\,dx\right)^\frac{N}{\gamma-N}\!\!+\|f^\varepsilon\|_{L^\gamma(B_{R})}^\frac{N\,\gamma}{\gamma-N}\right]\left\|\mathcal{G}_\varepsilon(\nabla u^\varepsilon)\right\|_{L^1(B_{R})}.
\end{split}
\]
We now take the same infinitesimal sequence \(\{\varepsilon_k\}_{k\in\mathbb{N}}\) as in the proof of Proposition \ref{prop:high}. 
By using again \eqref{pointwiseG}, the lower semicontinuity of the $L^\infty$ norm with respect to almost everywhere convergence, equation \eqref{eq-Vitali} 
and the fact that \(f^{\varepsilon_k}\) is defined from  \(f\) by convolution with a smooth kernel, the limit as $k$ goes to $\infty$ gives
\[
\begin{split} \left\|\mathcal{G}_0(\nabla U)\right\|_{L^{\infty}(B_\frac{R}{4})}
 &\le C\,\left[\left(\frac{4}{R}\right)^\frac{N\,\gamma}{\gamma-N}\,\left(\Gamma_1+\Gamma_2\,\int_{B_R} \mathcal{G}_0(\nabla U)\,dx\right)^\frac{N}{\gamma-N}+\|f\|_{L^\gamma(B_{R})}^\frac{N\,\gamma}{\gamma-N}\right]\,\left\|\mathcal{G}_0(\nabla U)\right\|_{L^1(B_{R})},
 \end{split}
\]
possibly for a different $C=C(N,p_N,p_1,\gamma)>0$.
This completes the proof.

\subsection{Proof of Theorem S}
We assume that \(1<p_1\leq \dots \leq p_N\leq 2\) and \(f\) verifying \eqref{gammadif}. We set for notational simplicity 
\[
\gamma=1+\frac{2}{p_1}.
\]
Consider the ball \(B_R \Subset B\) and let \(\eta\in C^{\infty}_0(B_{R/2})\) be such that 
\[
\eta\equiv 1 \mbox{ on } B_\frac{R}{4},\qquad 0\le \eta\le 1 \qquad \mbox{ and }\qquad |\nabla \eta|^2+|D^2\eta|\leq \frac{C_0}{R^2},
\]
for some \(C_0\) which depends only on \(N\). The choice of $\gamma$ entails the following estimates
\[
1+2\,\left(\frac{1}{p_1}-\frac{1}{p_N}\right)<\gamma \qquad \mbox{ and } \qquad \frac{2-p_1}{p_1}\frac{\gamma}{\gamma-2}\le \gamma.
\] 
Then, as a consequence of \eqref{aprioriV}, we have
\[
\begin{split}
\sum_{i=1}^N \int_{B_\frac{R}{4}} \Big|\nabla \mathcal{V}_{i,\varepsilon}\Big|^2 \,dx&\leq  \frac{C}{R^2}\,\int_{B_\frac{R}{2}} \mathcal{G}_\varepsilon(\nabla u^\varepsilon)^\gamma\,dx+C\left(\int_{B_\frac{R}{2}} \mathcal{G}_\varepsilon(\nabla u^\varepsilon)^\gamma\,dx\right)^\frac{\gamma-2}{\gamma}\,\|f^\varepsilon\|_{L^\gamma(B_R)}^2,
\end{split}
\]
where we also used that $\mathcal{G}_\varepsilon\ge 1$, by definition. We now rely again on \eqref{eq1734}, to estimate the terms containing $\mathcal{G}_\varepsilon^\gamma$. This estimate and Young's inequality with exponents $\gamma/2$ and $\gamma/(\gamma-2)$ give
\begin{equation}
\label{eq1920}
\sum_{i=1}^N \int_{B_\frac{R}{4}} \Big|\nabla \mathcal{V}_{i,\varepsilon}\Big|^2 \,dx\leq  \frac{C}{R^2}\,\left(\Gamma_1+\Gamma_2\,\int_{B_R} \mathcal{G}_\varepsilon(\nabla u^\varepsilon)\,dx\right)+C\,R^{\gamma-2}\,\|f^\varepsilon\|_{L^\gamma(B_R)}^\gamma,
\end{equation}
possibly for a different constant $C=C(N,p_N,p_1)>0$. From this estimate,  we deduce 
that the family \(\nabla\mathcal{V}_{i,\varepsilon}\) is uniformly bounded in \(L^{2}(B_{R/4})\). Moreover, by \eqref{g2} we have
\begin{equation}\label{serve}
\sqrt{g_{i,\varepsilon}''(t)}\le (\varepsilon+t^2)^\frac{p_i-2}{4} \le |t|^\frac{p_i-2}{2},\qquad \mbox{ for } t\not=0.
\end{equation}
Thus, by recalling the definition of $\mathcal{V}_{i,\varepsilon}$, we get
\[
\begin{split}
\int_{B_\frac{R}{4}} |\mathcal{V}_{i,\varepsilon}|^2\,dx=\int_{B_\frac{R}{4}} |V_{i,\varepsilon}(u^\varepsilon_{x_i})|^2\,dx&\le \left(\frac{2}{p_i}\right)^2\,\int_{B_\frac{R}{4}} |u^\varepsilon_{x_i}|^{p_i}\,dx,
\end{split}
\]
and the latter is uniformly bounded, thanks to Lemma \ref{lm:unif}. 
\par
Thus, by taking the same infinitesimal sequence \(\{\varepsilon_k\}_{k\geq 1}\) as in the proof of Proposition \ref{prop:high}, we have obtained that $\{\mathcal{V}_{i,\varepsilon_k}\}_{k\in\mathbb{N}}$ is a bounded sequence in $W^{1,2}(B_{R/4})$. By appealing to the Rellich-Kondra\v{s}ov Theorem, we can infer its convergence to a function $\mathcal{V}_i\in W^{1,2}(B_{R/4})$, weakly in $W^{1,2}(B_{R/4})$ and strongly in $L^2(B_{R/4})$ (up to a subsequence). By the weak lower semicontinuity of the \(L^{2}\) norm, \eqref{pointwiseG} and \eqref{eq-Vitali}, we obtain from \eqref{eq1920} that
\[
\sum_{i=1}^N \int_{B_\frac{R}{4}} \Big|\nabla \mathcal{V}_i\Big|^2 \,dx \leq  \frac{C}{R^2}\,\left(\Gamma_1+\Gamma_2\,\int_{B_R} \mathcal{G}_0(\nabla U)\,dx\right)+C\,R^{\gamma-2}\,\|f\|_{L^\gamma(B_R)}^\gamma,
\]
possibly for a different $C=C(N,p_N,p_1)>0$.
We claim that for every \(1\leq i \leq N\) and for almost every \(x\in B_{R/4}\), we have
\begin{equation}
\label{eq-id-Vi}
\mathcal{V}_i(x)=\frac{2}{p_i}\,\sqrt{p_i-1}\,|U_{x_i}(x)|^{(p_i-2)/2}\,U_{x_i}(x).
\end{equation}
Indeed, let us take $x\in B_{R/4}$ such that $u_{x_i}^{\varepsilon_k}$ converges to $U_{x_i}$ and such that $|U_{x_i}(x)|<+\infty$. Observe that the collection of these points has full measure in $B_{R/4}$. We then set
\[
M_i(x)=\sup_{k\in\mathbb{N}} |u_{x_i}^{\varepsilon_k}(x)|,
\]
which is finite, by construction.
Then, for every \(k\geq 0\)  we have
\[
\begin{split}
\left|\mathcal{V}_{i, \varepsilon_k}(x)-\frac{2}{p_i}\,\sqrt{p_i-1}\,|U_{x_i}(x)|^{\frac{p_i-2}{2}}\,U_{x_i}(x)\right|&= \left|V_{i,\varepsilon_k}(u^{\varepsilon_k}_{x_i}(x))- \frac{2}{p_i}\,\sqrt{p_i-1}\,|U_{x_i}(x)|^{\frac{p_i-2}{2}}\,U_{x_i}(x)\right|\\
&\le  \left|\int_{0}^{u_{x_i}^{\varepsilon_k}(x)}\left(\sqrt{g_{i,\varepsilon_k}''(\tau)}-\sqrt{(p_i-1)|\tau|^{p_i-2}}\right)\,d\tau\right|\\
&+\frac{2}{p_i}\,\sqrt{p_i-1}\,\left||u_{x_i}^{\varepsilon_k}(x)|^{\frac{p_i-2}{2}}\,u_{x_i}^{\varepsilon_k}(x)-|U_{x_i}(x)|^{\frac{p_i-2}{2}}\,U_{x_i}(x)\right|\\
&\leq \int_{0}^{M_i(x)}\left| \sqrt{g_{i,\varepsilon_k}''(\tau)}-\sqrt{(p_i-1)|\tau|^{p_i-2}}\right|\,d\tau\\
&+\frac{2}{p_i}\,\sqrt{p_i-1}\,\left||u_{x_i}^{\varepsilon_k}(x)|^{\frac{p_i-2}{2}}\,u_{x_i}^{\varepsilon_k}(x)-|U_{x_i}(x)|^{\frac{p_i-2}{2}}\,U_{x_i}(x)\right|.
\end{split}
\]
Thanks to \eqref{serve} 
one can apply the dominated convergence to conclude that the first term in the right-hand side converges to \(0\) when \(k\) goes to \(+\infty\). By also using that $u^{\varepsilon_k}_{x_i}(x)$ converges to $U_{x_i}(x)$, we finally get \eqref{eq-id-Vi}.
\par
By using the Chain Rule in Sobolev spaces, we also obtain that \(U_{x_i}\in W^{1,p_i}(B_{R/4})\) and satisfies the estimate claimed in the statement of Theorem S. The proof is now over.

\appendix

\section{A weak maximum principle}\label{appA}

Let \(G:\mathbb{R}^N\to [0,+\infty)\) be a  strictly convex function such that \(G(0)=0\). Let \(\zeta : \mathbb{R}\to \mathbb{R}\) be a  Lipschitz function, with the following property: there exists \(M>0\)  such that 
\[
\zeta(t)=\left\{\begin{array}{cc}
M,& \mbox{ if } t\ge M,\\
-M,& \mbox{ if } t\le -M.
\end{array}
\right.
\]
Given a ball \(B\subset \mathbb{R}^N\), $f\in L^1(B)$ and \(U\in W^{1,1}(B)\cap L^\infty(B)\) such that 
\[
\|U\|_{L^\infty(B)}\le M\qquad \mbox{ and }\qquad \int_B G(\nabla U)\,dx<+\infty,
\] 
we consider the functional
\[
\mathcal{F}(v)=\int_{B}\Big[G(\nabla v) + f\,\zeta(v)\Big]\,dx,\qquad \mbox{ for every }v\in U+W^{1,1}_0(B). 
\]
\begin{lm}\label{lm-bd-1}
If \(u\) is a minimum of \(\mathcal{F}\), then \(\|u\|_{L^{\infty}(B)}\leq M\). 
\end{lm}
\begin{proof}
We want to test the minimality of $u$ against the truncated function
\[
v:=\max\{-M, \min\{u, M\}\}.
\] 
By construction, we still have \(v\in U+W^{1,1}_0(B)\) and by minimality of \(u\), we get
\[
\mathcal{F}(u)\leq \mathcal{F}(v).
\]
By the properties of \(\zeta\) and the fact that \(G(0)=0\), we have
\[
\mathcal{F}(v)=\int_{\{|u|\leq M\}} \Big[G(\nabla u) + f\,\zeta(u)\Big]\,dx + \int_{\{u>M\}}f\,\zeta(M)\,dx +\int_{\{u<-M\}} f\,\zeta(-M)\,dx.  
\]
Hence, using that \(\zeta(u)=\zeta(M)\) when \(u\geq M\) and \(\zeta(u)=\zeta(-M)\) when \(u\leq -M\), by comparing the last two displays
we get
\[
\int_{\{|u|>M\}}G(\nabla u)\,dx= 0.  
\]
Since \(G\geq 0\) and \(G(\xi)=0\) if and only if \(\xi=0\), we deduce that \(\nabla u = 0\) almost everywhere on the set \(\{|u|>M\}\). It follows that \(\nabla u = \nabla v\) almost everywhere and since \(u=v=U\) on \(\partial B\), this implies that \(u=v\) almost everywhere in \(B\). In particular, \(|u|\leq M\) almost everywhere in \(B\).
\end{proof}

\end{document}